\documentclass{article}

\usepackage[a4paper, total={6in, 8in}]{geometry}
\usepackage{emptypage}
\usepackage[utf8]{inputenc}
\usepackage[T1]{fontenc}
\usepackage{lmodern}
\usepackage[english]{babel}
\usepackage{amsmath, amssymb}
\usepackage{amsthm}
\usepackage{longtable}
\usepackage{placeins}
\usepackage[font=small,labelfont=bf]{caption}
\usepackage{svg, rotating}

\usepackage{algorithm}
\usepackage{algpseudocode}
\algrenewcommand\algorithmicrequire{\textbf{Input:}}
\algrenewcommand\algorithmicensure{\textbf{Output:}}

\usepackage{enumerate}
\usepackage{xcolor}
\RequirePackage[colorlinks,allcolors=blue]{hyperref}
\RequirePackage{graphicx}
\usepackage[nottoc]{tocbibind}

\usepackage{titlesec}
\titleformat{\section}{\large\bfseries}{\thesection}{1em}{}

\usepackage[round]{natbib} 
\let\oldcite\cite
\renewcommand{\cite}[1]{\mbox{\oldcite{#1}}}

\newtheorem{theorem}{Theorem}
\newtheorem{lemma}{Lemma}
\newtheorem{corollary}{Corollary}

\theoremstyle{definition}

\newtheorem{remark}{Remark}
\newtheorem*{remark*}{Remark}

\DeclareMathOperator*{\argmin}{arg\,min}

\title{A Simple Bootstrap for Chatterjee's Rank Correlation}
\date{\today}
\author{Holger Dette\footnote{Faculty of Mathematics, Ruhr-Universität Bochum, \href{mailto:holger.dette@rub.de}{holger.dette@rub.de}} \and Marius Kroll\footnote{Faculty of Mathematics, Ruhr-Universität Bochum, \href{mailto:marius.kroll@rub.de}{marius.kroll@rub.de}}}
\begin{document}
\maketitle

\begin{abstract}
    We prove that  an $m$ out of $n$ bootstrap procedure for Chatterjee's rank correlation is consistent whenever asymptotic normality of Chatterjee's rank correlation can be established. In particular, we prove that $m$ out of $n$ bootstrap works for continuous as well as for discrete data with independent coordinates; furthermore, simulations indicate that it also performs well for discrete data with dependent coordinates, and that it outperforms alternative estimation methods. Consistency of the bootstrap is proved in the Kolmogorov  as well as  in the Wasserstein distance.\\[2mm]

\end{abstract}

\section{Introduction}
\label{sec:introduction}

A fundamental problem in statistics is to measure the strength of the  association between two random variables $X$ and $Y$, and many measures, such as Pearson’s and Spearman’s correlation, Kendall's $\tau$ and Gini's  $\gamma$, have been proposed for this purpose. Most of these measures
are very powerful to detect linear and monotone dependencies, but they
are less sensitive to non-monotone dependencies even in the case where $Y=f(X) $ for some deterministic function $f$.
Numerous authors have made proposals to address this problem
\citep[see, for example, the survey of][]{joseholmes2016}, where, as pointed out by \cite{chatterjee:2021}, a large part of the literature has its focus on testing independence.  In the same paper the author proposes  an estimate of a new measure, which
will be called Chatterjee's rank correlation throughout this paper.
Due to its simple form and appealing properties, it  has found considerable attention since its introduction
\citep[see][among others]{caobickel2020,debghosalsen2020,auddy_et_al:2021,linhan2021,gamboaetal2020,shidrtonhan2021b,shi_et_al:2022}.

To be precise, let  $\left(X_1, Y_1\right), \ldots, \left(X_n, Y_n\right)$ be  a sample of independent identically distributed (iid) observations, denote by $\left(X_{(1)},Y_{(1)}\right), \ldots ,\left(X_{(n)},Y_{(n)}\right)$ the rearranged data such that \mbox{$X_{(1)} \leq \ldots \leq X_{(n)}$} with ties broken at random and let $r_i$ be the rank of $Y_{(i)}$ $(i=1, \ldots , n)$.  \cite{chatterjee:2021} defines the rank coefficient $\xi_n$ by
$$
    \xi_n := \xi_n\left((X_1, Y_1), \ldots, (X_n, Y_n)\right) := 1 - \frac{3\sum_{i=1}^{n-1} |r_{i+1} - r_i|}{n^2 - 1}
$$
if there are no ties among the $Y_1, \ldots, Y_n$, and by
$$
    \xi_n := 1 - \frac{n \sum_{i=1}^{n-1} |r_{i+1} - r_i|}{2\sum_{i=1}^n l_i(n - l_i)}
$$
if  there are ties, where $l_i$ is the number of indices $j$ such that $Y_{(j)} \geq Y_{(i)}$.
\cite{chatterjee:2021}  shows that $\xi_n$ is a consistent estimator of the Dette-Siburg-Stoimenov dependence measure
\begin{align}
    \label{det3}
    \xi := \xi(X, Y) := \frac{\int \mathrm{Var}\left(\mathbb{E}\left[\textbf{1}_{[y, \infty)}(Y) ~|~ X\right]\right) ~\mathrm{d}\mathbb{P}^Y(y)}{\int \mathrm{Var}\left(\textbf{1}_{[y, \infty)}(Y)\right) ~\mathrm{d}\mathbb{P}^Y(y)},
\end{align}
where $(X,Y)$ is a generic vector with the same joint distribution as $(X_1, Y_1)$ and $\mathbb{P}^Y$ denotes the distribution of $Y$. This measure of dependence has been introduced by \cite{dette_siburg_stoimenov:2013}
for continuous distributions using  a representation in terms of copulas
and  is  independently considered by
\cite{chatterjee:2021} in the form \eqref{det3}. The measure $\xi$
has the  desirable properties mentioned above. In particular, it is $0$ if and only if $X$ and $Y$ are independent, and $1$ if and only if $Y = f(X)$ almost surely for some measurable function $f$.

In this article, we focus on one of the properties of Chatterjee's rank correlation, namely its asymptotic normality and estimation of the limiting variance, say $\sigma^2$, if it exists. To be precise,  define
$$
    T_n := \sqrt{n}(\xi_n - \mathbb{E} \xi_n).
$$
In a recent article, \cite{lin_han:2022} prove that, if $(X, Y)$ has a continuous distribution function and if $Y$ cannot be expressed as a measurable function of $X$ almost surely, then
$$
    \frac{T_n}{\mathrm{Var}(T_n)} \xrightarrow[n \to \infty]{\mathcal{D}} \mathcal{N}\left(0, 1\right).
$$
\cite{chatterjee:2021} proves a similar result under the assumption that $X$ and $Y$ are independent; in this case, it can furthermore be shown that $\mathrm{Var}(T_n)$ converges to $2/5$. The general result by \cite{lin_han:2022} requires an impressive amount of technical subtlety and allows for no such explicit calculation of the limit of $\mathrm{Var}(T_n)$.
On the other hand, if $(X, Y)$ follows a discrete distribution and $X$ and $Y$ are independent, then, by Theorem 2.2 in \cite{chatterjee:2021},
$$
    T_n \xrightarrow[n \to \infty]{\mathcal{D}} \mathcal{N}\left(0,\sigma^2\right)
$$
for some $\sigma^2 > 0$. Some other works such as \cite{auddy_et_al:2021}, \cite{shi_et_al:2022} derive similar limiting theorems under special assumptions on the kind of dependence between $X$ and $Y$. However, the limiting variance is generally an unknown quantity that must be estimated, unless $(X, Y)$ has a continuous distribution function \textit{and} $X$ and $Y$ are independent.

A popular tool in such situations is estimating the unknown limiting parameter via a bootstrap method. However, \cite{lin_han:2023} show that the usual bootstrap procedure fails for Chatterjee's rank correlation. Two alternatives remain: First, in the case of independence, \cite{chatterjee:2021} proposes an estimator for the limiting variance (which is only unknown in the case of discrete data). Second, in the case of continuous data, \cite{lin_han:2022} construct a consistent estimator for the unknown limiting variance and \cite{lin_han:2023} investigate the finite sample properties of  this estimator.

Our goal in this article is to establish an estimator for the unknown variance of the statistic  $T_n$ that  is applicable in any situation, regardless of the (in-)dependence between $X$ and $Y$ or the (dis-)continuity of the distribution function of $(X, Y)$. At this point, it is important to point out that, to the best of our knowledge, there is currently no weak convergence theory established for the case where $X$ and $Y$ are discrete and dependent. A more precise and careful formulation of our goal would therefore run as follows: Whenever the original statistic $T_n$ converges in distribution to some centred normal distribution with variance $\sigma^2$, we want our estimator to be consistent for this (usually unknown) limiting variance $\sigma^2$.

We achieve our goal via an $m$ out of $n$ bootstrap, going back to \cite{PolitisRomano1994}
\citep[see also][]{bickel_et_al:1997}.  Our main results are stated in Section  \ref{sec:mains} and show that an appropriately defined $m$ out of $n$ bootstrap can consistently estimate the limiting distribution of $T_n$, whenever $T_n$ is asymptotically normal. This is true in terms of the Kolmogorov distance as well as the Wasserstein distance of order $2$. As a consequence of this, we can use the $m$ out of $n$ bootstrap to estimate the unknown limiting variance of $T_n$. As usual in many applications of $m$ out of $n$ sampling, $m$ has to be of smaller order than $n$, but a stronger condition is required in the case of discrete distributions.

In Section \ref{sec:sims}, we present the results of a comprehensive simulation study to investigate the finite sample properties  of the proposed bootstrap estimator. In all these situations our estimator shows good performance. We also compare it to the estimation methods introduced by \cite{lin_han:2022} as well as the original estimator in \cite{chatterjee:2021}.
Some concluding remarks  are given in Section \ref{sec:conclusion} while  Section \ref{sec:proof} contains the proofs of our main results.

\section{Main Result}
\label{sec:mains}
From the observed sample $(X_1, Y_1), \ldots, (X_n, Y_n)$, we draw $m < n$ independent bootstrap observations $\left(X_1^*, Y_1^*\right), \ldots, \left(X_m^*, Y_m^*\right)$ without replacement and calculate Chatterjee's rank correlation based on this bootstrap sample, denoted by $\xi_m^*$. $\xi_m^*$ is a random variable, with randomness being introduced by the initial sample $(X_1, Y_1), \ldots, (X_n,Y_n)$, as well as the random sampling in the bootstrap process. We then estimate the unknown distribution of the statistic $T_m$ by the distribution of the bootstrap statistic
$$
    T_{m,n}^* := \sqrt{m}\left(\xi_m^* - \mathbb{E}\left[\xi_m^* ~|~ (X_1, Y_1), \ldots, (X_n, Y_n)\right]\right).
$$
Our main result shows that whenever $T_n$ weakly converges to some normal distribution, $T_{m,n}^{*}$ weakly converges in probability to the same limit. In the following statement $\Phi$ denotes the cumulative distribution function of the standard normal distribution.

\begin{theorem}
    \label{thm:mon_bs_u}
    Let $(X_k, Y_k)_{k \in \mathbb{N}}$ be an iid process and suppose that $m \to \infty$ as $n \to \infty$. Suppose further that one of the following two conditions is fulfilled:
    \begin{enumerate}[(i)]
        \item $T_n \xrightarrow[n \to \infty]{\mathcal{D}} \mathcal{N}\left(0,\sigma^2\right)$ for some $\sigma^2 > 0$ and $m = o\left(\sqrt{n}\right)$.
        \item $(X,Y)$ has a continuous distribution function, $Y$ cannot be expressed as a measurable function of $X$ almost surely and $m = o\left(n\right)$.
    \end{enumerate}
    Then it holds that
    $$
        \sup_{x \in \mathbb{R}} \left|F_{m,n}^*(x) - \Phi \Big ( {x \over \sigma}  \Big )\right| \xrightarrow[n \to \infty]{\mathbb{P}} 0,
    $$
    where
    $F_{m,n}^*$
    denotes the cumulative distribution functions of
    $T_{m,n}^*$, conditional on $(X_1, Y_1), \ldots, (X_n, Y_n)$, respectively.
\end{theorem}
\vspace{0mm}

\begin{remark} ~~ \label{rem1}
    \begin{enumerate}[(i)]
        \item It is shown in \cite{lin_han:2022} that condition (ii) in Theorem \ref{thm:mon_bs_u} implies asymptotic normality of $(\xi_n - \mathbb{E}\xi_n)/\sqrt{\mathrm{Var}(\xi_n)}$.
        \item The weak convergence condition in assumption $(i)$ in Theorem \ref{thm:mon_bs_u} is fulfilled if $X$ and $Y$ are independent and discrete \citep[see  Theorem 2.2 in][] {chatterjee:2021}.
        \item To the best of our knowledge, there are so far  no results about the weak convergence  of Chatterjee's rank correlation for discrete but dependent data. However, once such a result  is  established, Theorem \ref{thm:mon_bs_u} $(i)$ will be applicable to this case.
        \item Under assumption $(ii)$ in Theorem \ref{thm:mon_bs_u}, we can allow for the more relaxed growth rate $m = o(n)$ instead of $m = o\left(\sqrt{n}\right)$. This is a consequence of a bound on the variance of Chatterjee's rank correlation that has been establish by  \cite{lin_han:2022}. More specifically, these authors prove that under the assumptions of $(ii)$, it holds that
              $$
                  \mathrm{Var}(\xi_n((X_1, Y_1), \ldots, (X_n, Y_n))) = \mathcal{O}\left(\frac{1}{n}\right).
              $$
              In order to relax the growth rate $m = o\left(\sqrt{n}\right)$ in Theorem \ref{thm:mon_bs_u} $(i)$, one would need to find a similar bound for the variance of $\xi_n\left((X_{i_1}, Y_{i_1}), \ldots, (X_{i_n}, Y_{i_n})\right)$ in the case where $1 \leq i_1, \ldots, i_n \leq n$ is a collection of not necessarily distinct indices. However, as shown in the following section, even the rate of $m = o\left(\sqrt{n}\right)$ achieves good results in our simulation studies.
        \item An analogous result to Theorem \ref{thm:mon_bs_u} $(i)$ can be established for the measure of conditional dependence introduced by \cite{azadkia_chatterjee:2021}.
    \end{enumerate}
\end{remark}

We can establish a stronger version of Theorem \ref{thm:mon_bs_u} under an additional assumption on $T_n$. For this, recall the definition of the Wasserstein metric $d_p$ of order $p$ on the space of probability measures,
$$
    d_p^p(\eta, \xi) := \inf_{\gamma \in \Gamma} \int |x-y|^p ~\mathrm{d}\gamma(x,y),
$$
where  $\eta$ and $\xi$ are two probability measures and $\Gamma$ is the set of all joint distributions with marginals $\eta$ and $\xi$ (also called \textit{couplings} of $\eta$ and $\xi$). In the following discussion $\mathcal{L}(U)$ denotes the distribution of a random variable $U$.

\begin{theorem}
    \label{thm:konvergenz_wasserstein_metrik}
    Suppose that the conditions of Theorem \ref{thm:mon_bs_u} are satisfied and that the sequence $(T_n)_{n \in \mathbb{N}}$ is uniformly square-integrable.
    Then it holds that
    $$
        d_2\left(\mathcal{L}\left(T_{m,n}^*\right), \mathcal{N}\left(0, \sigma^2\right)\right)
        \xrightarrow[n \to \infty]{} 0.
    $$
\end{theorem}

\begin{remark}   \label{wassersteinremark}
    Under assumption $(ii)$ of Theorem \ref{thm:mon_bs_u} the uniform integrability condition of Theorem \ref{thm:konvergenz_wasserstein_metrik} is always satisfied. We prove this claim in Section \ref{subsec:proof_thm_wasserstein}.
\end{remark}

As convergence in $d_p$ is equivalent to weak convergence combined with convergence of the $p$-th moments or to weak convergence combined with uniform $p$-integrability  \citep[cf.\@ Theorem 6.9 in][]{villani:optimaltransport}, Theorem \ref{thm:konvergenz_wasserstein_metrik} tells us that the $m$ out of $n$ bootstrap is consistent in the Wasserstein metric $d_2$ if the same is true for the original statistic $T_n$. In other words, in terms of variance estimation, the $m$ out of $n$ bootstrap is as good as the original sequence $T_n$. This implies the following corollary.

\begin{corollary}
    \label{cor:pte-momente}
    Suppose that the conditions of Theorem \ref{thm:konvergenz_wasserstein_metrik} are satisfied. Let $f : \mathbb{R} \to \mathbb{R}$ be a continuous function such that $|f(x)| \leq C (1 + x^2)$ holds for all $x \in \mathbb{R}$ for some fixed $C > 0$. Then it holds that
    $$
        \mathbb{E}\left[f\left(T_{m,n}^*\right)\right] \xrightarrow[n \to \infty]{} \mathbb{E}\left[f(\sigma Z)\right],
    $$
    where $Z \sim \mathcal{N}(0,1)$ and $\sigma^2 > 0$ is the unknown limiting variance from Theorem \ref{thm:mon_bs_u}.
\end{corollary}

By choosing $f(x) = x^2$, we obtain convergence of the second moments. On the basis of this, we can estimate the unknown limiting variance itself by the second moments of our bootstrap statistic, i.e.\@ by
\begin{equation}
    \label{det0}
    \sigma_{m,n}^{*2} := \mathrm{Var}\left(\sqrt{m}\,\xi_m^* ~|~ (X_1, Y_1), \ldots, (X_n, Y_n)\right).
\end{equation}

\begin{remark}
    \label{det10}
    {\rm
        Let us consider the time complexity of the $m$ out of $n$ bootstrap for the choice  $m = \lfloor n^\gamma \rfloor$ for some parameter $0 < \gamma < 1$. By Remark 1.4 in \cite{chatterjee:2021} the calculation of the statistic from the  bootstrap sample (of size $m$) has a time complexity  of order $\mathcal{O}(m \log m)$. Moreover, drawing a sample of size $m$ from $n$ observations can be done by algorithms with time complexity of order $\mathcal{O}(m \log m)$ \citep[see, for example,][]{ting2021simple}. Thus the time complexity of the $m$ out of $n$ bootstrap is  of order $\mathcal{O}(B m \log m )$, where $B$ is the number of bootstrap iterations. Therefore, if $m = o(n)$, the time complexity is sub-linear in $n$ up to a logarithmic factor. Note that the number of bootstrap iterations is determined by the required precision for estimating the bootstrap distribution and, in applications, is usually chosen as some large but fixed number. Even if we would increase $B$ with the sample size such that $B= \lfloor n/m \rfloor $, the time complexity would be linear in $n$ up to a logarithmic factor (but we do not recommend this). Asymptotically, the time complexity of the $m$ out of $n$ bootstrap is therefore not worse than that of Chatterjee's rank correlation itself, which is $\mathcal{O}(n \log n)$.
    }
\end{remark}

\section{Finite sample properties}\label{sec:sims}

To assess the finite sample performance of the proposed $m$ out of $n$ bootstrap procedure, we conduct a simulation study
to provide answers to the following  questions:
\begin{enumerate}
    \item[(1)] What are good choices of $m$?
    \item[(2)]  How does the $m$ out of $n$ bootstrap compare to the `Lin-Han-estimator', i.e.\@ the variance estimator introduced by \cite{lin_han:2022}?
    \item[(3)] How well does the $m$ out of $n$ bootstrap perform for non-continuous sample data, where no other estimation methods are available?
    \item[(4)]  What is the time complexity of the $m$ out of $n$ bootstrap?
\end{enumerate}

\begin{table}[b]
    \centering\small
    \begin{tabular}{c c c c c c c c c c}
               &                & \multicolumn{2}{c}{Poisson} &            & \multicolumn{2}{c}{Gaussian} &       & \multicolumn{2}{c}{$t(3)$-distribution}                                       \\
        \cline{3-4} \cline{6-7} \cline{9-10}                                                                                                                                                      \\
        $\rho$ & \hspace{0.5cm} & $\xi$                       & $\sigma^2$ & \hspace{0.3cm}               & $\xi$ & $\sigma^2$                              & \hspace{0.3cm} & $\xi$ & $\sigma^2$ \\
        \hline
        0.0    & \hspace{0.5cm} & 0.00                        & 0.46       & \hspace{0.3cm}               & 0.00  & 0.40                                    & \hspace{0.3cm} & 0.00  & 0.42       \\
        0.3    & \hspace{0.5cm} & 0.06                        & 0.49       & \hspace{0.3cm}               & 0.05  & 0.46                                    & \hspace{0.3cm} & 0.06  & 0.50       \\
        0.5    & \hspace{0.5cm} & 0.18                        & 0.54       & \hspace{0.3cm}               & 0.14  & 0.51                                    & \hspace{0.3cm} & 0.15  & 0.58       \\
        0.7    & \hspace{0.5cm} & 0.23                        & 0.50       & \hspace{0.3cm}               & 0.30  & 0.49                                    & \hspace{0.3cm} & 0.30  & 0.59       \\
        0.9    & \hspace{0.5cm} & 0.65                        & 0.20       & \hspace{0.3cm}               & 0.58  & 0.23                                    & \hspace{0.3cm} & 0.58  & 0.32       \\
        \hline
    \end{tabular}
    \caption{\it  Population version of Chatterjee's rank correlation ($\xi$) and the limiting variance ($\sigma^2$) of the statistic $T_n$. These values have been generated by simulation for discrete and continuous models.}
    \label{tab:limiting-values}
\end{table}

\begin{figure}[t]
    \fontsize{8}{10}\selectfont
    \begingroup%
    \makeatletter%
    \providecommand\color[2][]{%
        \errmessage{(Inkscape) Color is used for the text in Inkscape, but the package 'color.sty' is not loaded}%
        \renewcommand\color[2][]{}%
    }%
    \providecommand\transparent[1]{%
        \errmessage{(Inkscape) Transparency is used (non-zero) for the text in Inkscape, but the package 'transparent.sty' is not loaded}%
        \renewcommand\transparent[1]{}%
    }%
    \providecommand\rotatebox[2]{#2}%
    \newcommand*\fsize{\dimexpr\f@size pt\relax}%
    \newcommand*\lineheight[1]{\fontsize{\fsize}{#1\fsize}\selectfont}%
    \ifx\svgwidth\undefined%
        \setlength{\unitlength}{431.25bp}%
        \ifx\svgscale\undefined%
            \relax%
        \else%
            \setlength{\unitlength}{\unitlength * \real{\svgscale}}%
        \fi%
    \else%
        \setlength{\unitlength}{\svgwidth}%
    \fi%
    \global\let\svgwidth\undefined%
    \global\let\svgscale\undefined%
    \makeatother%
    \begin{picture}(1,0.69565217)%
        \lineheight{1}%
        \setlength\tabcolsep{0pt}%
        \put(0,0){\includegraphics[width=\unitlength,page=1]{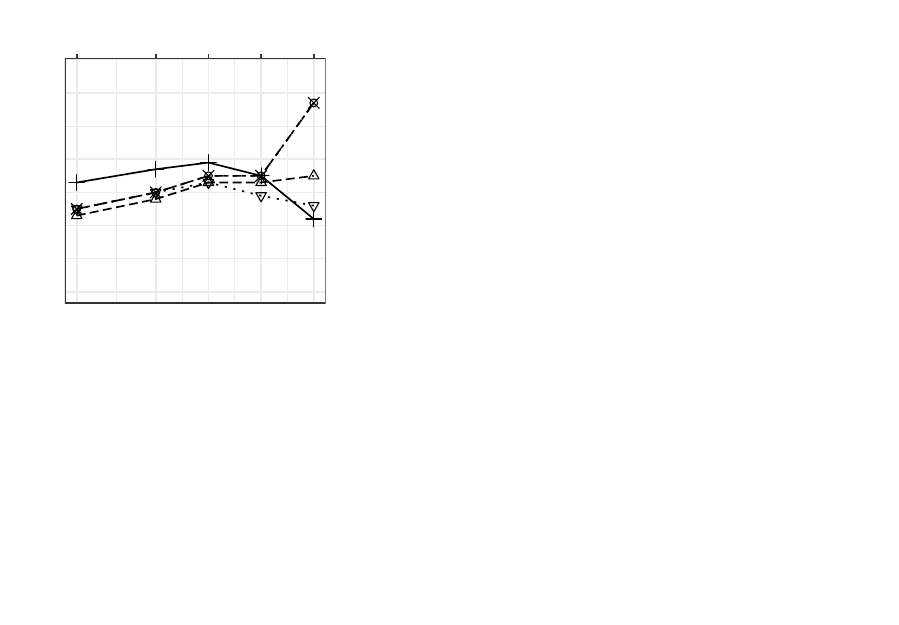}}%
        \put(0.08544348,0.63895652){\color[rgb]{0.30196078,0.30196078,0.30196078}\makebox(0,0)[t]{\lineheight{1.25}\smash{\begin{tabular}[t]{c}0.0\end{tabular}}}}%
        \put(0.1733913,0.63895652){\color[rgb]{0.30196078,0.30196078,0.30196078}\makebox(0,0)[t]{\lineheight{1.25}\smash{\begin{tabular}[t]{c}0.3\end{tabular}}}}%
        \put(0.23199999,0.63895652){\color[rgb]{0.30196078,0.30196078,0.30196078}\makebox(0,0)[t]{\lineheight{1.25}\smash{\begin{tabular}[t]{c}0.5\end{tabular}}}}%
        \put(0.29062609,0.63895652){\color[rgb]{0.30196078,0.30196078,0.30196078}\makebox(0,0)[t]{\lineheight{1.25}\smash{\begin{tabular}[t]{c}0.7\end{tabular}}}}%
        \put(0.34925219,0.63895652){\color[rgb]{0.30196078,0.30196078,0.30196078}\makebox(0,0)[t]{\lineheight{1.25}\smash{\begin{tabular}[t]{c}0.9\end{tabular}}}}%
        \put(0.06368696,0.36528694){\color[rgb]{0.30196078,0.30196078,0.30196078}\makebox(0,0)[rt]{\lineheight{1.25}\smash{\begin{tabular}[t]{r}0.0\end{tabular}}}}%
        \put(0.06368696,0.43909564){\color[rgb]{0.30196078,0.30196078,0.30196078}\makebox(0,0)[rt]{\lineheight{1.25}\smash{\begin{tabular}[t]{r}0.2\end{tabular}}}}%
        \put(0.06368696,0.51290434){\color[rgb]{0.30196078,0.30196078,0.30196078}\makebox(0,0)[rt]{\lineheight{1.25}\smash{\begin{tabular}[t]{r}0.4\end{tabular}}}}%
        \put(0.06368696,0.58671304){\color[rgb]{0.30196078,0.30196078,0.30196078}\makebox(0,0)[rt]{\lineheight{1.25}\smash{\begin{tabular}[t]{r}0.6\end{tabular}}}}%
        \put(0,0){\includegraphics[width=\unitlength,page=2]{plot_normdata_svg-tex.pdf}}%
        \put(0.03271304,0.49438261){\rotatebox{90}{\makebox(0,0)[t]{\lineheight{1.25}\smash{\begin{tabular}[t]{c}rRMSE\end{tabular}}}}}%
        \put(0.36245218,0.33793043){\makebox(0,0)[rt]{\lineheight{1.25}\smash{\begin{tabular}[t]{r}$n =50$\end{tabular}}}}%
        \put(0,0){\includegraphics[width=\unitlength,page=3]{plot_normdata_svg-tex.pdf}}%
        \put(0.39469565,0.63895652){\color[rgb]{0.30196078,0.30196078,0.30196078}\makebox(0,0)[t]{\lineheight{1.25}\smash{\begin{tabular}[t]{c}0.0\end{tabular}}}}%
        \put(0.4826261,0.63895652){\color[rgb]{0.30196078,0.30196078,0.30196078}\makebox(0,0)[t]{\lineheight{1.25}\smash{\begin{tabular}[t]{c}0.3\end{tabular}}}}%
        \put(0.54125218,0.63895652){\color[rgb]{0.30196078,0.30196078,0.30196078}\makebox(0,0)[t]{\lineheight{1.25}\smash{\begin{tabular}[t]{c}0.5\end{tabular}}}}%
        \put(0.59987825,0.63895652){\color[rgb]{0.30196078,0.30196078,0.30196078}\makebox(0,0)[t]{\lineheight{1.25}\smash{\begin{tabular}[t]{c}0.7\end{tabular}}}}%
        \put(0.65850437,0.63895652){\color[rgb]{0.30196078,0.30196078,0.30196078}\makebox(0,0)[t]{\lineheight{1.25}\smash{\begin{tabular}[t]{c}0.9\end{tabular}}}}%
        \put(0.67168696,0.33793043){\makebox(0,0)[rt]{\lineheight{1.25}\smash{\begin{tabular}[t]{r}$n =100$\end{tabular}}}}%
        \put(0,0){\includegraphics[width=\unitlength,page=4]{plot_normdata_svg-tex.pdf}}%
        \put(0.70394781,0.63895652){\color[rgb]{0.30196078,0.30196078,0.30196078}\makebox(0,0)[t]{\lineheight{1.25}\smash{\begin{tabular}[t]{c}0.0\end{tabular}}}}%
        \put(0.79187824,0.63895652){\color[rgb]{0.30196078,0.30196078,0.30196078}\makebox(0,0)[t]{\lineheight{1.25}\smash{\begin{tabular}[t]{c}0.3\end{tabular}}}}%
        \put(0.85050436,0.63895652){\color[rgb]{0.30196078,0.30196078,0.30196078}\makebox(0,0)[t]{\lineheight{1.25}\smash{\begin{tabular}[t]{c}0.5\end{tabular}}}}%
        \put(0.90913044,0.63895652){\color[rgb]{0.30196078,0.30196078,0.30196078}\makebox(0,0)[t]{\lineheight{1.25}\smash{\begin{tabular}[t]{c}0.7\end{tabular}}}}%
        \put(0.96775656,0.63895652){\color[rgb]{0.30196078,0.30196078,0.30196078}\makebox(0,0)[t]{\lineheight{1.25}\smash{\begin{tabular}[t]{c}0.9\end{tabular}}}}%
        \put(0.98093909,0.33793043){\makebox(0,0)[rt]{\lineheight{1.25}\smash{\begin{tabular}[t]{r}$n =500$\end{tabular}}}}%
        \put(0,0){\includegraphics[width=\unitlength,page=5]{plot_normdata_svg-tex.pdf}}%
        \put(0.06368696,0.05010434){\color[rgb]{0.30196078,0.30196078,0.30196078}\makebox(0,0)[rt]{\lineheight{1.25}\smash{\begin{tabular}[t]{r}0.0\end{tabular}}}}%
        \put(0.06368696,0.12391304){\color[rgb]{0.30196078,0.30196078,0.30196078}\makebox(0,0)[rt]{\lineheight{1.25}\smash{\begin{tabular}[t]{r}0.2\end{tabular}}}}%
        \put(0.06368696,0.19772174){\color[rgb]{0.30196078,0.30196078,0.30196078}\makebox(0,0)[rt]{\lineheight{1.25}\smash{\begin{tabular}[t]{r}0.4\end{tabular}}}}%
        \put(0.06368696,0.27153044){\color[rgb]{0.30196078,0.30196078,0.30196078}\makebox(0,0)[rt]{\lineheight{1.25}\smash{\begin{tabular}[t]{r}0.6\end{tabular}}}}%
        \put(0,0){\includegraphics[width=\unitlength,page=6]{plot_normdata_svg-tex.pdf}}%
        \put(0.03271304,0.1792){\rotatebox{90}{\makebox(0,0)[t]{\lineheight{1.25}\smash{\begin{tabular}[t]{c}rRMSE\end{tabular}}}}}%
        \put(0.36245218,0.0227478){\makebox(0,0)[rt]{\lineheight{1.25}\smash{\begin{tabular}[t]{r}$n =1000$\end{tabular}}}}%
        \put(0,0){\includegraphics[width=\unitlength,page=7]{plot_normdata_svg-tex.pdf}}%
        \put(0.37198261,0.65833044){\makebox(0,0)[t]{\lineheight{1.25}\smash{\begin{tabular}[t]{c}$\rho$\end{tabular}}}}%
        \put(0.83584345,0.65833044){\makebox(0,0)[t]{\lineheight{1.25}\smash{\begin{tabular}[t]{c}$\rho$\end{tabular}}}}%
        \put(0.67168696,0.0227478){\makebox(0,0)[rt]{\lineheight{1.25}\smash{\begin{tabular}[t]{r}$n =5000$\end{tabular}}}}%
        \put(0,0){\includegraphics[width=\unitlength,page=8]{plot_normdata_svg-tex.pdf}}%
        \put(0.76441741,0.25226087){\makebox(0,0)[lt]{\lineheight{1.25}\smash{\begin{tabular}[t]{l}Choice of $m$\end{tabular}}}}%
        \put(0,0){\includegraphics[width=\unitlength,page=9]{plot_normdata_svg-tex.pdf}}%
        \put(0.80399998,0.21994783){\makebox(0,0)[lt]{\lineheight{1.25}\smash{\begin{tabular}[t]{l}$m = n^{1/2}$\end{tabular}}}}%
        \put(0.80399998,0.18989566){\makebox(0,0)[lt]{\lineheight{1.25}\smash{\begin{tabular}[t]{l}$m = n^{1/4}$\end{tabular}}}}%
        \put(0.80399998,0.15984348){\makebox(0,0)[lt]{\lineheight{1.25}\smash{\begin{tabular}[t]{l}$m = n^{3/4}$\end{tabular}}}}%
        \put(0.80399998,0.12979131){\makebox(0,0)[lt]{\lineheight{1.25}\smash{\begin{tabular}[t]{l}Bickel \& Sakov\end{tabular}}}}%
        \put(0.80399998,0.09973914){\makebox(0,0)[lt]{\lineheight{1.25}\smash{\begin{tabular}[t]{l}Cluster\end{tabular}}}}%
    \end{picture}%
    \endgroup%
    \caption{\it Relative empirical root mean square error (rRMSE, \eqref{eq:rel_rmse}) of the bootstrap variance estimate \eqref{det0} for different methods of choosing the sample size $m = m(n)$. The data are generated from  a two-dimensional centred normal distribution with covariance matrix \eqref{det2} and different correlations $\rho$.}
    \label{fig:vergleich_m}
\end{figure}

\begin{figure}[t]
    \fontsize{8}{10}\selectfont
    \begingroup%
    \makeatletter%
    \providecommand\color[2][]{%
        \errmessage{(Inkscape) Color is used for the text in Inkscape, but the package 'color.sty' is not loaded}%
        \renewcommand\color[2][]{}%
    }%
    \providecommand\transparent[1]{%
        \errmessage{(Inkscape) Transparency is used (non-zero) for the text in Inkscape, but the package 'transparent.sty' is not loaded}%
        \renewcommand\transparent[1]{}%
    }%
    \providecommand\rotatebox[2]{#2}%
    \newcommand*\fsize{\dimexpr\f@size pt\relax}%
    \newcommand*\lineheight[1]{\fontsize{\fsize}{#1\fsize}\selectfont}%
    \ifx\svgwidth\undefined%
        \setlength{\unitlength}{431.25bp}%
        \ifx\svgscale\undefined%
            \relax%
        \else%
            \setlength{\unitlength}{\unitlength * \real{\svgscale}}%
        \fi%
    \else%
        \setlength{\unitlength}{\svgwidth}%
    \fi%
    \global\let\svgwidth\undefined%
    \global\let\svgscale\undefined%
    \makeatother%
    \begin{picture}(1,0.69565217)%
        \lineheight{1}%
        \setlength\tabcolsep{0pt}%
        \put(0,0){\includegraphics[width=\unitlength,page=1]{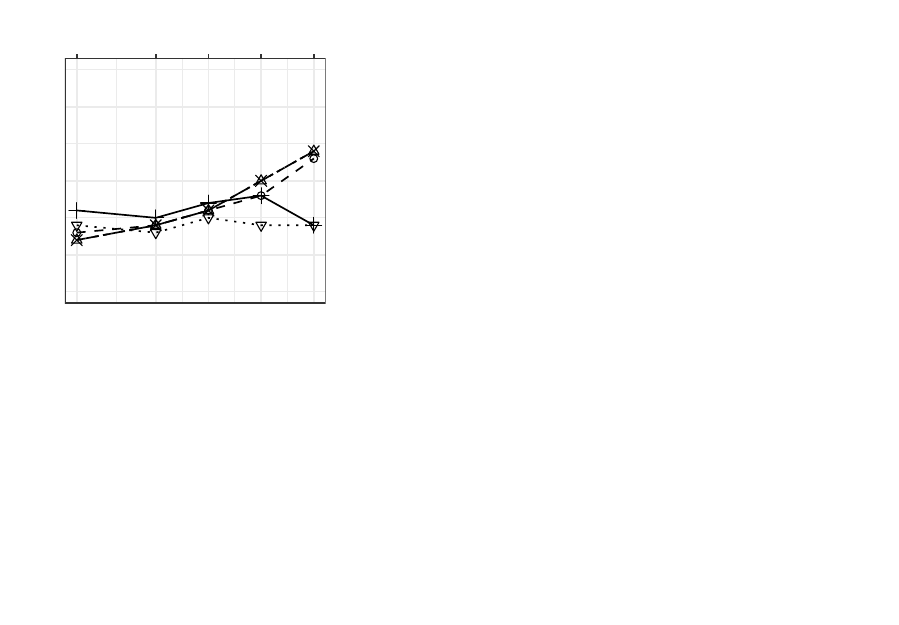}}%
        \put(0.08544348,0.63895652){\color[rgb]{0.30196078,0.30196078,0.30196078}\makebox(0,0)[t]{\lineheight{1.25}\smash{\begin{tabular}[t]{c}0.0\end{tabular}}}}%
        \put(0.1733913,0.63895652){\color[rgb]{0.30196078,0.30196078,0.30196078}\makebox(0,0)[t]{\lineheight{1.25}\smash{\begin{tabular}[t]{c}0.3\end{tabular}}}}%
        \put(0.23199999,0.63895652){\color[rgb]{0.30196078,0.30196078,0.30196078}\makebox(0,0)[t]{\lineheight{1.25}\smash{\begin{tabular}[t]{c}0.5\end{tabular}}}}%
        \put(0.29062609,0.63895652){\color[rgb]{0.30196078,0.30196078,0.30196078}\makebox(0,0)[t]{\lineheight{1.25}\smash{\begin{tabular}[t]{c}0.7\end{tabular}}}}%
        \put(0.34925219,0.63895652){\color[rgb]{0.30196078,0.30196078,0.30196078}\makebox(0,0)[t]{\lineheight{1.25}\smash{\begin{tabular}[t]{c}0.9\end{tabular}}}}%
        \put(0.06368696,0.40650436){\color[rgb]{0.30196078,0.30196078,0.30196078}\makebox(0,0)[rt]{\lineheight{1.25}\smash{\begin{tabular}[t]{r}0.0\end{tabular}}}}%
        \put(0.06368696,0.48892173){\color[rgb]{0.30196078,0.30196078,0.30196078}\makebox(0,0)[rt]{\lineheight{1.25}\smash{\begin{tabular}[t]{r}0.1\end{tabular}}}}%
        \put(0.06368696,0.57133912){\color[rgb]{0.30196078,0.30196078,0.30196078}\makebox(0,0)[rt]{\lineheight{1.25}\smash{\begin{tabular}[t]{r}0.2\end{tabular}}}}%
        \put(0,0){\includegraphics[width=\unitlength,page=2]{plot_normdata_covering_svg-tex.pdf}}%
        \put(0.03271304,0.49438261){\rotatebox{90}{\makebox(0,0)[t]{\lineheight{1.25}\smash{\begin{tabular}[t]{c}$0.95 -$Covering\end{tabular}}}}}%
        \put(0.36245218,0.33793043){\makebox(0,0)[rt]{\lineheight{1.25}\smash{\begin{tabular}[t]{r}$n = 50$\end{tabular}}}}%
        \put(0,0){\includegraphics[width=\unitlength,page=3]{plot_normdata_covering_svg-tex.pdf}}%
        \put(0.39469565,0.63895652){\color[rgb]{0.30196078,0.30196078,0.30196078}\makebox(0,0)[t]{\lineheight{1.25}\smash{\begin{tabular}[t]{c}0.0\end{tabular}}}}%
        \put(0.4826261,0.63895652){\color[rgb]{0.30196078,0.30196078,0.30196078}\makebox(0,0)[t]{\lineheight{1.25}\smash{\begin{tabular}[t]{c}0.3\end{tabular}}}}%
        \put(0.54125218,0.63895652){\color[rgb]{0.30196078,0.30196078,0.30196078}\makebox(0,0)[t]{\lineheight{1.25}\smash{\begin{tabular}[t]{c}0.5\end{tabular}}}}%
        \put(0.59987825,0.63895652){\color[rgb]{0.30196078,0.30196078,0.30196078}\makebox(0,0)[t]{\lineheight{1.25}\smash{\begin{tabular}[t]{c}0.7\end{tabular}}}}%
        \put(0.65850437,0.63895652){\color[rgb]{0.30196078,0.30196078,0.30196078}\makebox(0,0)[t]{\lineheight{1.25}\smash{\begin{tabular}[t]{c}0.9\end{tabular}}}}%
        \put(0.67168696,0.33793043){\makebox(0,0)[rt]{\lineheight{1.25}\smash{\begin{tabular}[t]{r}$n = 100$\end{tabular}}}}%
        \put(0,0){\includegraphics[width=\unitlength,page=4]{plot_normdata_covering_svg-tex.pdf}}%
        \put(0.70394781,0.63895652){\color[rgb]{0.30196078,0.30196078,0.30196078}\makebox(0,0)[t]{\lineheight{1.25}\smash{\begin{tabular}[t]{c}0.0\end{tabular}}}}%
        \put(0.79187824,0.63895652){\color[rgb]{0.30196078,0.30196078,0.30196078}\makebox(0,0)[t]{\lineheight{1.25}\smash{\begin{tabular}[t]{c}0.3\end{tabular}}}}%
        \put(0.85050436,0.63895652){\color[rgb]{0.30196078,0.30196078,0.30196078}\makebox(0,0)[t]{\lineheight{1.25}\smash{\begin{tabular}[t]{c}0.5\end{tabular}}}}%
        \put(0.90913044,0.63895652){\color[rgb]{0.30196078,0.30196078,0.30196078}\makebox(0,0)[t]{\lineheight{1.25}\smash{\begin{tabular}[t]{c}0.7\end{tabular}}}}%
        \put(0.96775656,0.63895652){\color[rgb]{0.30196078,0.30196078,0.30196078}\makebox(0,0)[t]{\lineheight{1.25}\smash{\begin{tabular}[t]{c}0.9\end{tabular}}}}%
        \put(0.98093909,0.33793043){\makebox(0,0)[rt]{\lineheight{1.25}\smash{\begin{tabular}[t]{r}$n = 500$\end{tabular}}}}%
        \put(0,0){\includegraphics[width=\unitlength,page=5]{plot_normdata_covering_svg-tex.pdf}}%
        \put(0.06368696,0.09130435){\color[rgb]{0.30196078,0.30196078,0.30196078}\makebox(0,0)[rt]{\lineheight{1.25}\smash{\begin{tabular}[t]{r}0.0\end{tabular}}}}%
        \put(0.06368696,0.17373912){\color[rgb]{0.30196078,0.30196078,0.30196078}\makebox(0,0)[rt]{\lineheight{1.25}\smash{\begin{tabular}[t]{r}0.1\end{tabular}}}}%
        \put(0.06368696,0.25615651){\color[rgb]{0.30196078,0.30196078,0.30196078}\makebox(0,0)[rt]{\lineheight{1.25}\smash{\begin{tabular}[t]{r}0.2\end{tabular}}}}%
        \put(0,0){\includegraphics[width=\unitlength,page=6]{plot_normdata_covering_svg-tex.pdf}}%
        \put(0.03271304,0.1792){\rotatebox{90}{\makebox(0,0)[t]{\lineheight{1.25}\smash{\begin{tabular}[t]{c}$0.95 -$Covering\end{tabular}}}}}%
        \put(0.36245218,0.0227478){\makebox(0,0)[rt]{\lineheight{1.25}\smash{\begin{tabular}[t]{r}$n = 1000$\end{tabular}}}}%
        \put(0,0){\includegraphics[width=\unitlength,page=7]{plot_normdata_covering_svg-tex.pdf}}%
        \put(0.37198261,0.65833044){\makebox(0,0)[t]{\lineheight{1.25}\smash{\begin{tabular}[t]{c}$\rho$\end{tabular}}}}%
        \put(0.83584345,0.65833044){\makebox(0,0)[t]{\lineheight{1.25}\smash{\begin{tabular}[t]{c}$\rho$\end{tabular}}}}%
        \put(0.67168696,0.0227478){\makebox(0,0)[rt]{\lineheight{1.25}\smash{\begin{tabular}[t]{r}$n = 5000$\end{tabular}}}}%
        \put(0,0){\includegraphics[width=\unitlength,page=8]{plot_normdata_covering_svg-tex.pdf}}%
        \put(0.76441741,0.25226087){\makebox(0,0)[lt]{\lineheight{1.25}\smash{\begin{tabular}[t]{l}Choice of $m$\end{tabular}}}}%
        \put(0,0){\includegraphics[width=\unitlength,page=9]{plot_normdata_covering_svg-tex.pdf}}%
        \put(0.80399998,0.21994783){\makebox(0,0)[lt]{\lineheight{1.25}\smash{\begin{tabular}[t]{l}$m = n^{1/2}$\end{tabular}}}}%
        \put(0.80399998,0.18989566){\makebox(0,0)[lt]{\lineheight{1.25}\smash{\begin{tabular}[t]{l}$m = n^{1/4}$\end{tabular}}}}%
        \put(0.80399998,0.15984348){\makebox(0,0)[lt]{\lineheight{1.25}\smash{\begin{tabular}[t]{l}$m = n^{3/4}$\end{tabular}}}}%
        \put(0.80399998,0.12979131){\makebox(0,0)[lt]{\lineheight{1.25}\smash{\begin{tabular}[t]{l}Bickel \& Sakov\end{tabular}}}}%
        \put(0.80399998,0.09973914){\makebox(0,0)[lt]{\lineheight{1.25}\smash{\begin{tabular}[t]{l}Cluster\end{tabular}}}}%
    \end{picture}%
    \endgroup%
    \caption{\it Difference of the empirical covering rate to the nominal $95\%$ level of the bootstrap confidence interval \eqref{det1} for different methods of choosing the sample size $m = m(n)$. The data are generated from  a two-dimensional centred normal distribution with covariance matrix \eqref{det2} and different correlations $\rho$.}
    \label{fig:vergleich_m_covering}
\end{figure}

\begin{figure}[t]
    \fontsize{8}{10}\selectfont
    \begingroup%
    \makeatletter%
    \providecommand\color[2][]{%
        \errmessage{(Inkscape) Color is used for the text in Inkscape, but the package 'color.sty' is not loaded}%
        \renewcommand\color[2][]{}%
    }%
    \providecommand\transparent[1]{%
        \errmessage{(Inkscape) Transparency is used (non-zero) for the text in Inkscape, but the package 'transparent.sty' is not loaded}%
        \renewcommand\transparent[1]{}%
    }%
    \providecommand\rotatebox[2]{#2}%
    \newcommand*\fsize{\dimexpr\f@size pt\relax}%
    \newcommand*\lineheight[1]{\fontsize{\fsize}{#1\fsize}\selectfont}%
    \ifx\svgwidth\undefined%
        \setlength{\unitlength}{431.25bp}%
        \ifx\svgscale\undefined%
            \relax%
        \else%
            \setlength{\unitlength}{\unitlength * \real{\svgscale}}%
        \fi%
    \else%
        \setlength{\unitlength}{\svgwidth}%
    \fi%
    \global\let\svgwidth\undefined%
    \global\let\svgscale\undefined%
    \makeatother%
    \begin{picture}(1,0.69565217)%
        \lineheight{1}%
        \setlength\tabcolsep{0pt}%
        \put(0,0){\includegraphics[width=\unitlength,page=1]{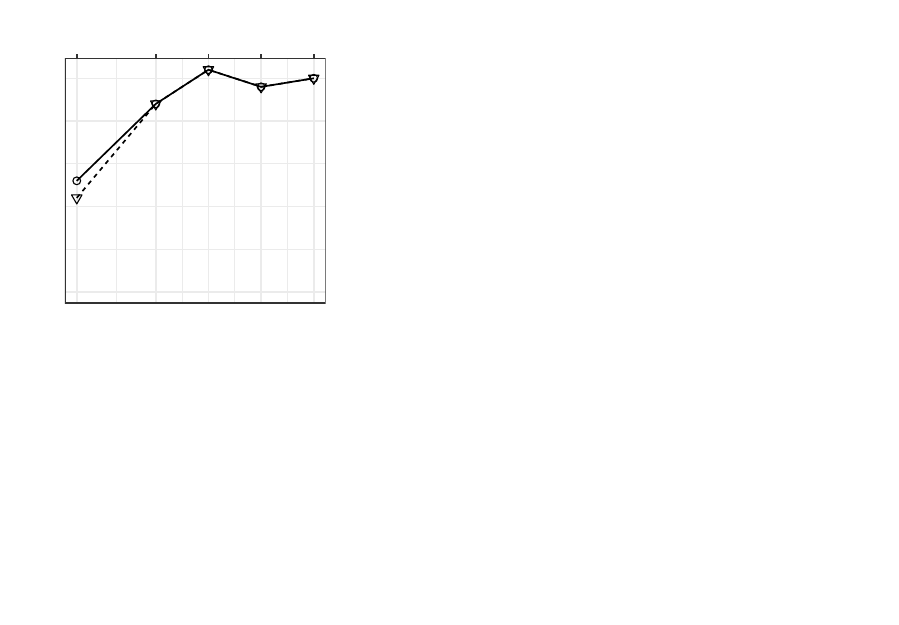}}%
        \put(0.08544348,0.63895652){\color[rgb]{0.30196078,0.30196078,0.30196078}\makebox(0,0)[t]{\lineheight{1.25}\smash{\begin{tabular}[t]{c}0.0\end{tabular}}}}%
        \put(0.1733913,0.63895652){\color[rgb]{0.30196078,0.30196078,0.30196078}\makebox(0,0)[t]{\lineheight{1.25}\smash{\begin{tabular}[t]{c}0.3\end{tabular}}}}%
        \put(0.23199999,0.63895652){\color[rgb]{0.30196078,0.30196078,0.30196078}\makebox(0,0)[t]{\lineheight{1.25}\smash{\begin{tabular}[t]{c}0.5\end{tabular}}}}%
        \put(0.29062609,0.63895652){\color[rgb]{0.30196078,0.30196078,0.30196078}\makebox(0,0)[t]{\lineheight{1.25}\smash{\begin{tabular}[t]{c}0.7\end{tabular}}}}%
        \put(0.34925219,0.63895652){\color[rgb]{0.30196078,0.30196078,0.30196078}\makebox(0,0)[t]{\lineheight{1.25}\smash{\begin{tabular}[t]{c}0.9\end{tabular}}}}%
        \put(0.06368696,0.36528694){\color[rgb]{0.30196078,0.30196078,0.30196078}\makebox(0,0)[rt]{\lineheight{1.25}\smash{\begin{tabular}[t]{r}0.0\end{tabular}}}}%
        \put(0.06368696,0.46038261){\color[rgb]{0.30196078,0.30196078,0.30196078}\makebox(0,0)[rt]{\lineheight{1.25}\smash{\begin{tabular}[t]{r}0.1\end{tabular}}}}%
        \put(0.06368696,0.55549566){\color[rgb]{0.30196078,0.30196078,0.30196078}\makebox(0,0)[rt]{\lineheight{1.25}\smash{\begin{tabular}[t]{r}0.2\end{tabular}}}}%
        \put(0,0){\includegraphics[width=\unitlength,page=2]{poiplot_svg-tex.pdf}}%
        \put(0.03271304,0.49438261){\rotatebox{90}{\makebox(0,0)[t]{\lineheight{1.25}\smash{\begin{tabular}[t]{c}rRMSE\end{tabular}}}}}%
        \put(0.36245218,0.33793043){\makebox(0,0)[rt]{\lineheight{1.25}\smash{\begin{tabular}[t]{r}$n =50$\end{tabular}}}}%
        \put(0,0){\includegraphics[width=\unitlength,page=3]{poiplot_svg-tex.pdf}}%
        \put(0.39469565,0.63895652){\color[rgb]{0.30196078,0.30196078,0.30196078}\makebox(0,0)[t]{\lineheight{1.25}\smash{\begin{tabular}[t]{c}0.0\end{tabular}}}}%
        \put(0.4826261,0.63895652){\color[rgb]{0.30196078,0.30196078,0.30196078}\makebox(0,0)[t]{\lineheight{1.25}\smash{\begin{tabular}[t]{c}0.3\end{tabular}}}}%
        \put(0.54125218,0.63895652){\color[rgb]{0.30196078,0.30196078,0.30196078}\makebox(0,0)[t]{\lineheight{1.25}\smash{\begin{tabular}[t]{c}0.5\end{tabular}}}}%
        \put(0.59987825,0.63895652){\color[rgb]{0.30196078,0.30196078,0.30196078}\makebox(0,0)[t]{\lineheight{1.25}\smash{\begin{tabular}[t]{c}0.7\end{tabular}}}}%
        \put(0.65850437,0.63895652){\color[rgb]{0.30196078,0.30196078,0.30196078}\makebox(0,0)[t]{\lineheight{1.25}\smash{\begin{tabular}[t]{c}0.9\end{tabular}}}}%
        \put(0.67168696,0.33793043){\makebox(0,0)[rt]{\lineheight{1.25}\smash{\begin{tabular}[t]{r}$n =100$\end{tabular}}}}%
        \put(0,0){\includegraphics[width=\unitlength,page=4]{poiplot_svg-tex.pdf}}%
        \put(0.70394781,0.63895652){\color[rgb]{0.30196078,0.30196078,0.30196078}\makebox(0,0)[t]{\lineheight{1.25}\smash{\begin{tabular}[t]{c}0.0\end{tabular}}}}%
        \put(0.79187824,0.63895652){\color[rgb]{0.30196078,0.30196078,0.30196078}\makebox(0,0)[t]{\lineheight{1.25}\smash{\begin{tabular}[t]{c}0.3\end{tabular}}}}%
        \put(0.85050436,0.63895652){\color[rgb]{0.30196078,0.30196078,0.30196078}\makebox(0,0)[t]{\lineheight{1.25}\smash{\begin{tabular}[t]{c}0.5\end{tabular}}}}%
        \put(0.90913044,0.63895652){\color[rgb]{0.30196078,0.30196078,0.30196078}\makebox(0,0)[t]{\lineheight{1.25}\smash{\begin{tabular}[t]{c}0.7\end{tabular}}}}%
        \put(0.96775656,0.63895652){\color[rgb]{0.30196078,0.30196078,0.30196078}\makebox(0,0)[t]{\lineheight{1.25}\smash{\begin{tabular}[t]{c}0.9\end{tabular}}}}%
        \put(0.98093909,0.33793043){\makebox(0,0)[rt]{\lineheight{1.25}\smash{\begin{tabular}[t]{r}$n =500$\end{tabular}}}}%
        \put(0,0){\includegraphics[width=\unitlength,page=5]{poiplot_svg-tex.pdf}}%
        \put(0.06368696,0.05010434){\color[rgb]{0.30196078,0.30196078,0.30196078}\makebox(0,0)[rt]{\lineheight{1.25}\smash{\begin{tabular}[t]{r}0.0\end{tabular}}}}%
        \put(0.06368696,0.14519998){\color[rgb]{0.30196078,0.30196078,0.30196078}\makebox(0,0)[rt]{\lineheight{1.25}\smash{\begin{tabular}[t]{r}0.1\end{tabular}}}}%
        \put(0.06368696,0.24029568){\color[rgb]{0.30196078,0.30196078,0.30196078}\makebox(0,0)[rt]{\lineheight{1.25}\smash{\begin{tabular}[t]{r}0.2\end{tabular}}}}%
        \put(0,0){\includegraphics[width=\unitlength,page=6]{poiplot_svg-tex.pdf}}%
        \put(0.03271304,0.1792){\rotatebox{90}{\makebox(0,0)[t]{\lineheight{1.25}\smash{\begin{tabular}[t]{c}rRMSE\end{tabular}}}}}%
        \put(0.36245218,0.0227478){\makebox(0,0)[rt]{\lineheight{1.25}\smash{\begin{tabular}[t]{r}$n =1000$\end{tabular}}}}%
        \put(0,0){\includegraphics[width=\unitlength,page=7]{poiplot_svg-tex.pdf}}%
        \put(0.37198261,0.65833044){\makebox(0,0)[t]{\lineheight{1.25}\smash{\begin{tabular}[t]{c}$\rho$\end{tabular}}}}%
        \put(0.83584345,0.65833044){\makebox(0,0)[t]{\lineheight{1.25}\smash{\begin{tabular}[t]{c}$\rho$\end{tabular}}}}%
        \put(0.67168696,0.0227478){\makebox(0,0)[rt]{\lineheight{1.25}\smash{\begin{tabular}[t]{r}$n =5000$\end{tabular}}}}%
        \put(0,0){\includegraphics[width=\unitlength,page=8]{poiplot_svg-tex.pdf}}%
        \put(0.76441741,0.20718262){\makebox(0,0)[lt]{\lineheight{1.25}\smash{\begin{tabular}[t]{l}Choice of $m$\end{tabular}}}}%
        \put(0,0){\includegraphics[width=\unitlength,page=9]{poiplot_svg-tex.pdf}}%
        \put(0.80399998,0.17486954){\makebox(0,0)[lt]{\lineheight{1.25}\smash{\begin{tabular}[t]{l}Bickel \& Sakov\end{tabular}}}}%
        \put(0.80399998,0.14481737){\makebox(0,0)[lt]{\lineheight{1.25}\smash{\begin{tabular}[t]{l}Cluster\end{tabular}}}}%
    \end{picture}%
    \endgroup%
    \caption{\it Relative empirical root mean square error (rRMSE, \eqref{eq:rel_rmse}) of the bootstrap variance estimate \eqref{det0} for the cluster method as well as the Bickel-Sakov rule of choosing the sample size $m = m(n)$. The data follow the Poisson model \eqref{det4}.}
    \label{fig:bickel_cluster_poisson}
\end{figure}

\begin{figure}[t]
    \fontsize{8}{10}\selectfont
    \begingroup%
    \makeatletter%
    \providecommand\color[2][]{%
        \errmessage{(Inkscape) Color is used for the text in Inkscape, but the package 'color.sty' is not loaded}%
        \renewcommand\color[2][]{}%
    }%
    \providecommand\transparent[1]{%
        \errmessage{(Inkscape) Transparency is used (non-zero) for the text in Inkscape, but the package 'transparent.sty' is not loaded}%
        \renewcommand\transparent[1]{}%
    }%
    \providecommand\rotatebox[2]{#2}%
    \newcommand*\fsize{\dimexpr\f@size pt\relax}%
    \newcommand*\lineheight[1]{\fontsize{\fsize}{#1\fsize}\selectfont}%
    \ifx\svgwidth\undefined%
        \setlength{\unitlength}{431.25bp}%
        \ifx\svgscale\undefined%
            \relax%
        \else%
            \setlength{\unitlength}{\unitlength * \real{\svgscale}}%
        \fi%
    \else%
        \setlength{\unitlength}{\svgwidth}%
    \fi%
    \global\let\svgwidth\undefined%
    \global\let\svgscale\undefined%
    \makeatother%
    \begin{picture}(1,0.69565217)%
        \lineheight{1}%
        \setlength\tabcolsep{0pt}%
        \put(0,0){\includegraphics[width=\unitlength,page=1]{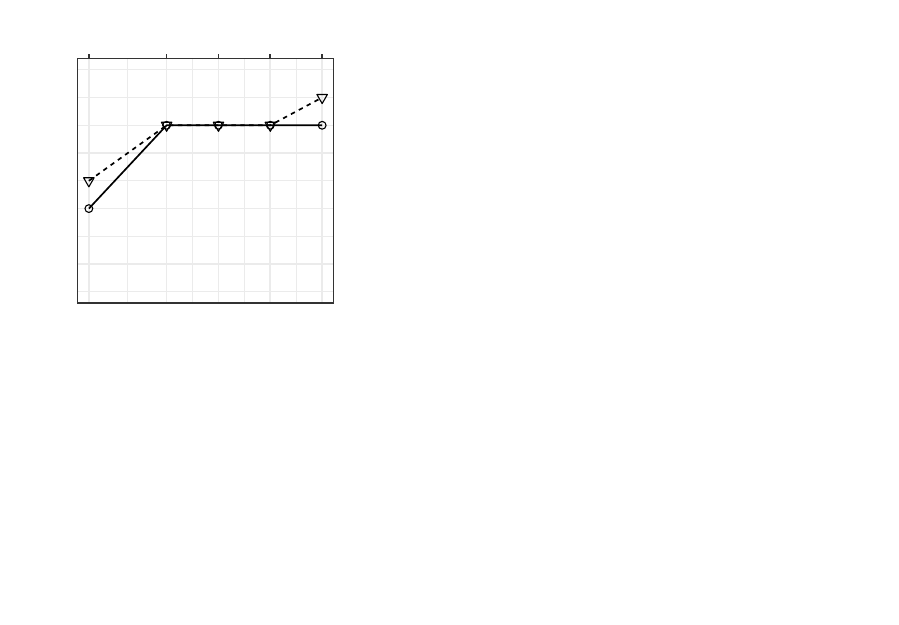}}%
        \put(0.09892174,0.63895652){\color[rgb]{0.30196078,0.30196078,0.30196078}\makebox(0,0)[t]{\lineheight{1.25}\smash{\begin{tabular}[t]{c}0.0\end{tabular}}}}%
        \put(0.18546087,0.63895652){\color[rgb]{0.30196078,0.30196078,0.30196078}\makebox(0,0)[t]{\lineheight{1.25}\smash{\begin{tabular}[t]{c}0.3\end{tabular}}}}%
        \put(0.24316523,0.63895652){\color[rgb]{0.30196078,0.30196078,0.30196078}\makebox(0,0)[t]{\lineheight{1.25}\smash{\begin{tabular}[t]{c}0.5\end{tabular}}}}%
        \put(0.30086957,0.63895652){\color[rgb]{0.30196078,0.30196078,0.30196078}\makebox(0,0)[t]{\lineheight{1.25}\smash{\begin{tabular}[t]{c}0.7\end{tabular}}}}%
        \put(0.3585739,0.63895652){\color[rgb]{0.30196078,0.30196078,0.30196078}\makebox(0,0)[t]{\lineheight{1.25}\smash{\begin{tabular}[t]{c}0.9\end{tabular}}}}%
        \put(0.07735652,0.3961913){\color[rgb]{0.30196078,0.30196078,0.30196078}\makebox(0,0)[rt]{\lineheight{1.25}\smash{\begin{tabular}[t]{r}-0.02\end{tabular}}}}%
        \put(0.07735652,0.45801739){\color[rgb]{0.30196078,0.30196078,0.30196078}\makebox(0,0)[rt]{\lineheight{1.25}\smash{\begin{tabular}[t]{r}0.00\end{tabular}}}}%
        \put(0.07735652,0.51982609){\color[rgb]{0.30196078,0.30196078,0.30196078}\makebox(0,0)[rt]{\lineheight{1.25}\smash{\begin{tabular}[t]{r}0.02\end{tabular}}}}%
        \put(0.07735652,0.58165217){\color[rgb]{0.30196078,0.30196078,0.30196078}\makebox(0,0)[rt]{\lineheight{1.25}\smash{\begin{tabular}[t]{r}0.04\end{tabular}}}}%
        \put(0,0){\includegraphics[width=\unitlength,page=2]{poiplot_covering_svg-tex.pdf}}%
        \put(0.03271304,0.49438261){\rotatebox{90}{\makebox(0,0)[t]{\lineheight{1.25}\smash{\begin{tabular}[t]{c}$0.95 -$Covering\end{tabular}}}}}%
        \put(0.37156521,0.33793043){\makebox(0,0)[rt]{\lineheight{1.25}\smash{\begin{tabular}[t]{r}$n = 50$\end{tabular}}}}%
        \put(0,0){\includegraphics[width=\unitlength,page=3]{poiplot_covering_svg-tex.pdf}}%
        \put(0.40360001,0.63895652){\color[rgb]{0.30196078,0.30196078,0.30196078}\makebox(0,0)[t]{\lineheight{1.25}\smash{\begin{tabular}[t]{c}0.0\end{tabular}}}}%
        \put(0.49015652,0.63895652){\color[rgb]{0.30196078,0.30196078,0.30196078}\makebox(0,0)[t]{\lineheight{1.25}\smash{\begin{tabular}[t]{c}0.3\end{tabular}}}}%
        \put(0.54786085,0.63895652){\color[rgb]{0.30196078,0.30196078,0.30196078}\makebox(0,0)[t]{\lineheight{1.25}\smash{\begin{tabular}[t]{c}0.5\end{tabular}}}}%
        \put(0.60556524,0.63895652){\color[rgb]{0.30196078,0.30196078,0.30196078}\makebox(0,0)[t]{\lineheight{1.25}\smash{\begin{tabular}[t]{c}0.7\end{tabular}}}}%
        \put(0.66326957,0.63895652){\color[rgb]{0.30196078,0.30196078,0.30196078}\makebox(0,0)[t]{\lineheight{1.25}\smash{\begin{tabular}[t]{c}0.9\end{tabular}}}}%
        \put(0.67624347,0.33793043){\makebox(0,0)[rt]{\lineheight{1.25}\smash{\begin{tabular}[t]{r}$n = 100$\end{tabular}}}}%
        \put(0,0){\includegraphics[width=\unitlength,page=4]{poiplot_covering_svg-tex.pdf}}%
        \put(0.70829563,0.63895652){\color[rgb]{0.30196078,0.30196078,0.30196078}\makebox(0,0)[t]{\lineheight{1.25}\smash{\begin{tabular}[t]{c}0.0\end{tabular}}}}%
        \put(0.79485219,0.63895652){\color[rgb]{0.30196078,0.30196078,0.30196078}\makebox(0,0)[t]{\lineheight{1.25}\smash{\begin{tabular}[t]{c}0.3\end{tabular}}}}%
        \put(0.85255652,0.63895652){\color[rgb]{0.30196078,0.30196078,0.30196078}\makebox(0,0)[t]{\lineheight{1.25}\smash{\begin{tabular}[t]{c}0.5\end{tabular}}}}%
        \put(0.91026091,0.63895652){\color[rgb]{0.30196078,0.30196078,0.30196078}\makebox(0,0)[t]{\lineheight{1.25}\smash{\begin{tabular}[t]{c}0.7\end{tabular}}}}%
        \put(0.96796525,0.63895652){\color[rgb]{0.30196078,0.30196078,0.30196078}\makebox(0,0)[t]{\lineheight{1.25}\smash{\begin{tabular}[t]{c}0.9\end{tabular}}}}%
        \put(0.98093909,0.33793043){\makebox(0,0)[rt]{\lineheight{1.25}\smash{\begin{tabular}[t]{r}$n = 500$\end{tabular}}}}%
        \put(0,0){\includegraphics[width=\unitlength,page=5]{poiplot_covering_svg-tex.pdf}}%
        \put(0.07735652,0.08100867){\color[rgb]{0.30196078,0.30196078,0.30196078}\makebox(0,0)[rt]{\lineheight{1.25}\smash{\begin{tabular}[t]{r}-0.02\end{tabular}}}}%
        \put(0.07735652,0.14281738){\color[rgb]{0.30196078,0.30196078,0.30196078}\makebox(0,0)[rt]{\lineheight{1.25}\smash{\begin{tabular}[t]{r}0.00\end{tabular}}}}%
        \put(0.07735652,0.2046435){\color[rgb]{0.30196078,0.30196078,0.30196078}\makebox(0,0)[rt]{\lineheight{1.25}\smash{\begin{tabular}[t]{r}0.02\end{tabular}}}}%
        \put(0.07735652,0.26645219){\color[rgb]{0.30196078,0.30196078,0.30196078}\makebox(0,0)[rt]{\lineheight{1.25}\smash{\begin{tabular}[t]{r}0.04\end{tabular}}}}%
        \put(0,0){\includegraphics[width=\unitlength,page=6]{poiplot_covering_svg-tex.pdf}}%
        \put(0.03271304,0.1792){\rotatebox{90}{\makebox(0,0)[t]{\lineheight{1.25}\smash{\begin{tabular}[t]{c}$0.95 -$Covering\end{tabular}}}}}%
        \put(0.37156521,0.0227478){\makebox(0,0)[rt]{\lineheight{1.25}\smash{\begin{tabular}[t]{r}$n = 1000$\end{tabular}}}}%
        \put(0,0){\includegraphics[width=\unitlength,page=7]{poiplot_covering_svg-tex.pdf}}%
        \put(0.38109566,0.65833044){\makebox(0,0)[t]{\lineheight{1.25}\smash{\begin{tabular}[t]{c}$\rho$\end{tabular}}}}%
        \put(0.83812176,0.65833044){\makebox(0,0)[t]{\lineheight{1.25}\smash{\begin{tabular}[t]{c}$\rho$\end{tabular}}}}%
        \put(0.67624347,0.0227478){\makebox(0,0)[rt]{\lineheight{1.25}\smash{\begin{tabular}[t]{r}$n = 5000$\end{tabular}}}}%
        \put(0,0){\includegraphics[width=\unitlength,page=8]{poiplot_covering_svg-tex.pdf}}%
        \put(0.76669566,0.20718262){\makebox(0,0)[lt]{\lineheight{1.25}\smash{\begin{tabular}[t]{l}Choice of $m$\end{tabular}}}}%
        \put(0,0){\includegraphics[width=\unitlength,page=9]{poiplot_covering_svg-tex.pdf}}%
        \put(0.80627824,0.17486954){\makebox(0,0)[lt]{\lineheight{1.25}\smash{\begin{tabular}[t]{l}Bickel \& Sakov\end{tabular}}}}%
        \put(0.80627824,0.14481737){\makebox(0,0)[lt]{\lineheight{1.25}\smash{\begin{tabular}[t]{l}Cluster\end{tabular}}}}%
    \end{picture}%
    \endgroup%
    \caption{\it Difference of the empirical covering rate to the nominal $95\%$ level of the bootstrap confidence interval \eqref{det1} for the cluster method as well as the Bickel-Sakov rule of choosing the sample size $m = m(n)$. The data follow the Poisson model \eqref{det4}.}
    \label{fig:bickel_cluster_poisson_covering}
\end{figure}

\begin{table}[t]
    \centering\small
    \begin{tabular}{c r c c c c c c c c c}
               &      & \multicolumn{4}{c}{$m$ out of $n$} &       & \multicolumn{4}{c}{Lin-Han-estimator}                                                              \\
        \cline{3-6}\cline{8-11}                                                                                                                                         \\
        $\rho$ & $n$  & RMSE                               & rRMSE & Covering                              & Length & \hspace{0.0cm} & RMSE & rRMSE & Covering & Length \\
        \hline
        0.0    & 50   & 0.10                               & 0.25  & 0.91                                  & 0.31   &                & 0.49 & 1.23  & 0.43     & (0.20) \\
               & 100  & 0.06                               & 0.15  & 0.93                                  & 0.23   &                & 0.41 & 1.03  & 0.51     & (0.15) \\
               & 500  & 0.03                               & 0.08  & 0.95                                  & 0.11   &                & 0.24 & 0.60  & 0.82     & (0.09) \\
               & 1000 & 0.02                               & 0.05  & 0.94                                  & 0.08   &                & 0.17 & 0.43  & 0.89     & (0.07) \\
               & 5000 & 0.01                               & 0.03  & 0.95                                  & 0.03   &                & 0.08 & 0.20  & 0.94     & 0.03   \\
        \hline
        0.3    & 50   & 0.14                               & 0.30  & 0.92                                  & 0.32   &                & 0.45 & 0.98  & 0.22     & (0.09) \\
               & 100  & 0.10                               & 0.22  & 0.92                                  & 0.24   &                & 0.40 & 0.87  & 0.39     & (0.11) \\
               & 500  & 0.04                               & 0.09  & 0.93                                  & 0.11   &                & 0.25 & 0.54  & 0.80     & (0.10) \\
               & 1000 & 0.03                               & 0.07  & 0.93                                  & 0.08   &                & 0.17 & 0.37  & 0.91     & 0.08   \\
               & 5000 & 0.02                               & 0.04  & 0.94                                  & 0.04   &                & 0.07 & 0.15  & 0.94     & 0.04   \\
        \hline
        0.5    & 50   & 0.17                               & 0.33  & 0.90                                  & 0.33   &                & 0.49 & 0.96  & 0.11     & (0.05) \\
               & 100  & 0.12                               & 0.24  & 0.92                                  & 0.25   &                & 0.45 & 0.88  & 0.28     & (0.07) \\
               & 500  & 0.05                               & 0.10  & 0.93                                  & 0.12   &                & 0.25 & 0.49  & 0.83     & (0.10) \\
               & 1000 & 0.04                               & 0.08  & 0.94                                  & 0.09   &                & 0.17 & 0.33  & 0.90     & 0.08   \\
               & 5000 & 0.02                               & 0.04  & 0.93                                  & 0.04   &                & 0.07 & 0.14  & 0.92     & 0.04   \\
        \hline
        0.7    & 50   & 0.14                               & 0.29  & 0.91                                  & 0.32   &                & 0.47 & 0.96  & 0.02     & (0.01) \\
               & 100  & 0.10                               & 0.20  & 0.93                                  & 0.24   &                & 0.44 & 0.89  & 0.14     & (0.04) \\
               & 500  & 0.04                               & 0.08  & 0.93                                  & 0.12   &                & 0.23 & 0.47  & 0.81     & (0.09) \\
               & 1000 & 0.03                               & 0.06  & 0.95                                  & 0.08   &                & 0.15 & 0.31  & 0.91     & 0.08   \\
               & 5000 & 0.02                               & 0.04  & 0.93                                  & 0.04   &                & 0.06 & 0.12  & 0.93     & 0.04   \\
        \hline
        0.9    & 50   & 0.06                               & 0.26  & 0.91                                  & 0.25   &                & 0.24 & 1.04  & 0.00     & (0.00) \\
               & 100  & 0.05                               & 0.22  & 0.94                                  & 0.18   &                & 0.24 & 1.04  & 0.00     & (0.00) \\
               & 500  & 0.02                               & 0.09  & 0.94                                  & 0.09   &                & 0.18 & 0.78  & 0.52     & (0.04) \\
               & 1000 & 0.02                               & 0.09  & 0.96                                  & 0.06   &                & 0.12 & 0.52  & 0.82     & (0.05) \\
               & 5000 & 0.01                               & 0.04  & 0.95                                  & 0.03   &                & 0.04 & 0.17  & 0.95     & 0.03   \\
        \hline
    \end{tabular}
    \caption{\it Empirical root mean square error (RMSE) and relative root mean square error (rRMSE) \eqref{eq:rel_rmse} of the bootstrap variance estimate \eqref{det0}
        (left part) and the variance estimate of \cite{lin_han:2022} (right part). The table also shows the
        empirical coverage probability (covering)  and length (length)  of the  $95\%$-bootstrap confidence interval obtained by the different methods. The data are generated from a two-dimensional centred normal distribution with covariance matrix \eqref{det2} and $m$ chosen according to the cluster rule. Numbers in parentheses indicate confidence intervals with an empirical coverage probability of less than $90\%$, which we consider to fail.}
    \label{tab:normal}
\end{table}

\begin{table}[t]
    \centering\small
    \begin{tabular}{c r c c c c c c c c c}
               &      & \multicolumn{4}{c}{$m$ out of $n$} &       & \multicolumn{4}{c}{Lin-Han-estimator}                                                                \\
        \cline{3-6}\cline{8-11}                                                                                                                                           \\
        $\rho$ & $n$  & RMSE                               & rRMSE & Covering                              & Length & \hspace{0.0cm} & RMSE & rRMSE & Covering & Length   \\
        \hline
        0.0    & 50   & 0.12                               & 0.29  & 0.91                                  & 0.30   &                & 0.53 & 1.26  & 0.43     & (0.20)   \\
               & 100  & 0.08                               & 0.19  & 0.93                                  & 0.23   &                & 0.43 & 1.02  & 0.53     & (0.16)   \\
               & 500  & 0.04                               & 0.10  & 0.92                                  & 0.11   &                & 0.24 & 0.57  & 0.81     & (0.10)   \\
               & 1000 & 0.03                               & 0.07  & 0.91                                  & 0.08   &                & 0.19 & 0.45  & 0.87     & (0.08)   \\
               & 5000 & 0.02                               & 0.05  & 0.84                                  & (0.04) &                & 0.08 & 0.19  & 0.81     & (0.04)   \\
        \hline
        0.3    & 50   & 0.18                               & 0.36  & 0.92                                  & 0.32   &                & 0.49 & 0.98  & 0.24     & (0.11)   \\
               & 100  & 0.13                               & 0.26  & 0.92                                  & 0.24   &                & 0.44 & 0.88  & 0.43     & (0.13)   \\
               & 500  & 0.06                               & 0.12  & 0.92                                  & 0.12   &                & 0.26 & 0.52  & 0.82     & (0.10)   \\
               & 1000 & 0.05                               & 0.10  & 0.93                                  & 0.08   &                & 0.18 & 0.36  & 0.89     & (0.08)   \\
               & 5000 & 0.02                               & 0.04  & 0.93                                  & 0.04   &                & 0.08 & 0.16  & 0.93     & 0.04     \\
        \hline
        0.5    & 50   & 0.23                               & 0.40  & 0.87                                  & (0.33) &                & 0.55 & 0.95  & 0.15     & (0.06)   \\
               & 100  & 0.17                               & 0.29  & 0.89                                  & (0.25) &                & 0.50 & 0.86  & 0.35     & (0.10)   \\
               & 500  & 0.08                               & 0.14  & 0.92                                  & 0.12   &                & 0.25 & 0.43  & 0.84     & (0.11)   \\
               & 1000 & 0.06                               & 0.10  & 0.94                                  & 0.09   &                & 0.17 & 0.29  & 0.91     & 0.09     \\
               & 5000 & 0.04                               & 0.07  & 0.95                                  & 0.04   &                & 0.07 & 0.12  & 0.94     & 0.04     \\
        \hline
        0.7    & 50   & 0.24                               & 0.41  & 0.88                                  & (0.33) &                & 0.58 & 0.98  & 0.05     & (0.02)   \\
               & 100  & 0.18                               & 0.31  & 0.91                                  & 0.25   &                & 0.54 & 0.92  & 0.18     & (0.05)   \\
               & 500  & 0.10                               & 0.17  & 0.94                                  & 0.12   &                & 0.26 & 0.44  & 0.86     & (0.11)   \\
               & 1000 & 0.08                               & 0.14  & 0.94                                  & 0.09   &                & 0.15 & 0.25  & 0.93     & 0.09     \\
               & 5000 & 0.05                               & 0.08  & 0.93                                  & 0.04   &                & 0.06 & 0.10  & 0.92     & 0.04     \\
        \hline
        0.9    & 50   & 0.11                               & 0.34  & 0.90                                  & 0.27   &                & 0.32 & 1.00  & < 0.01   & (< 0.01) \\
               & 100  & 0.08                               & 0.25  & 0.91                                  & 0.20   &                & 0.32 & 1.00  & 0.01     & (< 0.01) \\
               & 500  & 0.04                               & 0.13  & 0.94                                  & 0.09   &                & 0.21 & 0.66  & 0.70     & (0.06)   \\
               & 1000 & 0.03                               & 0.09  & 0.94                                  & 0.07   &                & 0.12 & 0.38  & 0.89     & (0.06)   \\
               & 5000 & 0.02                               & 0.06  & 0.94                                  & 0.03   &                & 0.04 & 0.13  & 0.93     & 0.03     \\
        \hline
    \end{tabular}
    \caption{\it Empirical root mean square error (RMSE) and relative root mean square error (rRMSE) \eqref{eq:rel_rmse} of the bootstrap variance estimate \eqref{det0}
        (left part) and the variance estimate of \cite{lin_han:2022} (right part). The table also shows the
        empirical coverage probability (covering)  and length (length)  of the  $95\%$-bootstrap confidence interval obtained by the different methods. The data are generated from a two-dimensional $t(3)$ distribution with scale matrix \eqref{eq:t3_correlation} and $m$ chosen according to the cluster rule. Numbers in parentheses indicate confidence intervals with an empirical coverage probability of less than $90\%$, which we consider to fail. The entry $< 0.01$ indicates a value smaller than $0.01$.}
    \label{tab:t3}
\end{table}

\begin{table}[t]
    \centering\small
    \begin{tabular}{c r c c c c c c}
        $\rho$ & $n$  & RMSE (BS) & Coverage (BS) & Length (BS) & RMSE (C) & Coverage  (C) & Length (C) \\
        \hline
        0.0    & 50   & 0.07      & 0.92          & 0.35        & 0.03     & 0.96          & 0.38       \\
               & 100  & 0.05      & 0.93          & 0.25        & 0.02     & 0.96          & 0.27       \\
               & 500  & 0.02      & 0.94          & 0.12        & 0.01     & 0.95          & 0.12       \\
               & 1000 & 0.02      & 0.94          & 0.08        & 0.01     & 0.95          & 0.08       \\
               & 5000 & 0.01      & 0.95          & 0.04        & < 0.01   & 0.95          & 0.04       \\
        \hline
    \end{tabular}
    \caption{ \it
        Empirical root mean square error (RMSE) of the bootstrap variance estimate \eqref{det0}  and empirical coverage probabilities  of the  $ 95\%$-bootstrap confidence interval \eqref{det1}. The data are generated according to the model \eqref{det4} with $\rho=0$. The columns  denoted by (BS)  and (C) correspond to the $m$ out of $n$ bootstrap and Chatterjee's variance estimator \citep[see the  definition of  $\hat{\tau}_n^2$ after Theorem 2.2 in][]{chatterjee:2021}, respectively. The entry $< 0.01$ indicates a value smaller than $0.01$.}
    \label{tab:comparison-chatterjee-bs}
\end{table}

To investigate the first question, we compare five different rules of choosing $m$. The first three are independent of the observed data and are given by $m = \lfloor n^\gamma\rfloor$, where $\gamma \in \{1/4, 1/2, 3/4\}$. The fourth rule is an adaptive rule introduced by \cite{BickelSakov2008} which consists of the following steps:
\begin{enumerate}
    \item[(i)] For each $j = 0, 1, 2, \ldots$ define $m_j := \lceil q^j n\rceil$, where $q \in (0,1)$ is some fixed parameter, and let $L_j^*$ denote the distribution of $T_{m_j,n}^*$ conditional on the observed sample $(X_1, Y_1), \ldots, (X_n, Y_n)$.
    \item[(ii)] For some metric $\rho$ consistent with weak convergence let $J$ be the smallest $j$ which minimises the distance $\rho\left(L_j^*, L_{j+1}^*\right)$.
    \item[(iii)] Use $L_J^*$ as the bootstrap approximation.
\end{enumerate}
According to \cite{BickelSakov2008}, this rule is motivated by the observation that $L_j^*$ changes only slightly whenever $j$ is in the right range of values. Under some theoretical assumptions, this rule can be proven to be asymptotically optimal.
However, at least in finite samples, it does not always perform best. This sub-optimal performance can occur if there are some distributions among the entire collection $\{ L_j^* \colon  j \geq 1 \} $ which are close to each other but far away from the limiting distribution, in our case  a  $\mathcal{N}\left(0, \sigma^2\right)$ distribution. In other words
we observe some finite sample problems
in the rule of \cite{BickelSakov2008} if the distributions $\{ L_j^* \colon  j \geq 1 \} $ form `clusters' far away from the limiting distribution.
To remedy this behaviour, we implement yet another data-driven rule, which we use as the fifth and final rule in the  simulation study.

Starting with a collection $m_1, \ldots, m_K$ of admissible rules (i.e.\@ each $m_k$ satisfies the assumptions of Theorem \ref{thm:mon_bs_u}), we proceed as follows:
\begin{enumerate}
    \item[(i)] For each $j = 1, \ldots, K$ let $L_j^*$ denote the distribution of $T_{m_j,n}^*$ conditional on the observed sample $(X_1, Y_1), \ldots, (X_n, Y_n)$.
    \item[(ii)] For some metric $\rho$ consistent with weak convergence let $J \in \argmin_{1 \leq j \leq K} \sum_{k=1}^K\rho\left(L_j^*, L_{k}^*\right)$.
    \item[(iii)]  Use $L_J^*$  bootstrap approximation.
\end{enumerate}
In contrast to the Bickel-Sakov-rule, we know that each $L_j^*$ converges to the true limiting distribution, because all rules are assumed to be admissible. Thus, heuristically, most of the  distributions $\{ L_j^* \colon j=1, \ldots , K \} $
should be close  to the limiting distribution with respect to the distance in $\rho$. If some of the rules form clusters far away from $\mathcal{N}\left(0, \sigma^2\right)$ as described above, they should accordingly also be far away from most other distributions among the $\{ L_j^* \colon j=1, \ldots , K \}  $, and thus will not be chosen by our rule. On the other hand, the law $L_J^*$ that does get chosen must be close to most other distributions $\{L_j^* \colon j=1, \ldots , K \} $, making it likely that it is also close to the limiting  $\mathcal{N}\left(0, \sigma^2\right)$ distribution. We will refer to this method as the cluster rule.

Throughout this section, we choose as  metric $\rho$ the Kolmogorov distance, and approximate the distribution of $L_j^*$ by Monte Carlo with $B = 2000$ bootstrap repetitions. When implementing the cluster method, we use $m = \lfloor n^\gamma\rfloor, \gamma = 0.4, 0.45, \ldots, 0.9,$ as our set of admissible rules $\{m_1, \ldots, m_K\}$. The parameter $q$ in the Bickel-Sakov rule is chosen as $q = 1/2$.

We compare the five different rules by how useful they are for approximating the unknown limiting variance. More precisely, we simulate data of different sample sizes $n \in \{50, 100, 500, 1000, 5000\}$, which we use  to
\begin{enumerate}[(a)]
    \item estimate the unknown limiting variance of $T_n$ using the $m$ out of $n$ bootstrap estimator $\sigma^{*2}_{m,n} $ and, based on this variance estimate,
    \item construct an asymptotic 95 \%-confidence interval of the form
          \begin{align} \label{det1}
              \big ( \xi_n -  z_{0.975} \cdot \sigma_{m,n}^*/\sqrt{n},~
              \xi_n +  z_{0.975} \cdot \sigma_{m,n}^*/\sqrt{n} \big ),
          \end{align}
          where $z_{0.975}$ is the lower $97.5 \%$ quantile of the standard normal distribution.
\end{enumerate}
We use $B = 2000$ bootstrap repetitions in each simulation run to  determine the variance estimator $\sigma_{m,n}^{* 2}$ and $M = 1000$ simulation runs are used to calculate the empirical root mean square error (RMSE) of the variance estimate and the empirical covering rates of the bootstrap confidence intervals. To make the RMSE more interpretable we normalise it with the true limiting variance, obtaining the relative RMSE (rRMSE)
\begin{equation}
    \label{eq:rel_rmse}
    \textrm{rRMSE} := \textrm{RMSE}/\sigma^2,
\end{equation}
where $\sigma^2$ is the limiting variance under the data generating model. The model-specific limiting variances as well as the Dette-Siburg-Stoimenov measures of dependence $\xi$ are given in Table \ref{tab:limiting-values}. To obtain these values, we generate $M = 50\,000$ iid copies of $\xi_n$ under each model with sample size $n = 100 \, 000$. The estimations of $\xi$ and $\sigma^2$ are then chosen as the average and $n$ times the sample variance of the $50\,000$ copies of $\xi_n$.

The rRMSE of the five different rules for choosing $m$ are compared in Figure \@ \ref{fig:vergleich_m}. The data are  generated according to $(X_k, Y_k) \overset{iid}{\sim} \mathcal{N}\left(\mu, \Sigma\right)$ with expectation $\mu = (0, 0)^T$, covariance matrix
\begin{align}
    \label{det2}
    \Sigma = \begin{pmatrix}
                 1    & \rho \\
                 \rho & 1
             \end{pmatrix}
\end{align}
and correlation $\rho \in \{0, 0.3, 0.5, 0.7, 0.9\}$. The most apparent effect is the unstableness of most of the rules at $\rho = 0.9$. At this correlation, $Y_k$ is close to being a linear function of $X_k$, in which case no asymptotics for Chatterjee's rank correlation exist. The rules $n^{1/4}$ and $n^{3/4}$ are effected by this at least  in some cases, the Bickel-Sakov-rule in all cases. The rule $m = \sqrt{n}$ does not seem to be as susceptible to this `almost degenerate' case, but its performance is not optimal for smaller values of $\rho$. On the other hand, the cluster rule described above is stable even at $\rho = 0.9$ while also attaining the lowest or second lowest rRMSE in all cases.

A comparison of the different rules in the context of confience interval construction is given in Figure  \ref{fig:vergleich_m_covering}. Here we plot the difference of the empirical confience interval to the nominal $95\%$ level. The performance of the different rules compared in terms of the rRMSE basically translates to the performance in terms of the covering rate. However,  we point out that the confidence intervals constructed by the cluster method are
the only ones which stay within an acceptable range of the nominal $95\%$ level for all values of $n$ and $\rho$.
The lengths of the confidence intervals obtained by  the five rules only  show little variation, and therefore we do not report them here.

To further compare the Bickel-Sakov rule with the cluster method of choosing $m$, we consider a second data example, namely
$$
    X_k \overset{iid}{\sim} \mathrm{Poisson}(2),
$$
and, for an independent copy $(Z_k)_{k \in \mathbb{N}}$ of $(X_k)_{k \in \mathbb{N}}$,
\begin{align}
    \label{det4}
    Y_k := \tau X_k + (1 - \tau) Z_k,
\end{align}
for different parameters $0 \leq \tau < 1$. Note that in this model the correlation between $X$ and $Y$ is given by
\begin{equation}
    \label{eq:poimodel_korrelation}
    \rho := {\rm Corr} (X,Y) = \frac{\tau}{\sqrt{\tau  ^2 + (1-\tau )^2}}
\end{equation}
and we choose the values of $\tau$ such that we obtain $\rho \in \{0, 0.3, 0.5, 0.7, 0.9\}$ to make both  models comparable.

In the discrete case we have to make a minor adjustment, since Chatterjee's rank correlation is not defined if $Y_1 = \ldots = Y_n$. We therefore discard any bootstrap sample $(X_1^*, Y_1^*), \ldots, (X_m^*, Y_m^*)$ for which $Y_1^* = \ldots = Y_m^*$ and  replace it  by newly drawn `valid' bootstrap samples, thus ensuring that we end up with the desired total of $B  = 2000$ bootstrap repetitions. It should be noted, however, that this issue never came up for sample sizes larger than $100$.

The comparison of the Bickel-Sakov rule and  the cluster method is given in Figures  \ref{fig:bickel_cluster_poisson} and \ref{fig:bickel_cluster_poisson_covering}. The cluster rule outperforms the Bickel-Sakov rule in terms of the rRMSE. The jump at $\rho = 0.9$ is still present, but less extreme than in the Gaussian model; furthermore, the cluster rule manages to decrease this effect with growing $n$. The covering rates are generally closer to the nominal level of $95 \%$ when constructed according to the cluster rule but the differences are not substantial.

Summarising the findings from Figure  \ref{fig:vergleich_m} -  \ref{fig:bickel_cluster_poisson_covering} leads us to the  following conclusions. Although not always the best,   the rule $m = \sqrt{n}$  is a reasonable choice for $m$. In  the Gaussian example it is even robust in the edge case $\rho = 0.9$.
The cluster rule of choosing $m$ is not infallible, but its overall behaviour is  most convincing among the  methods  under consideration. Finally, the Bickel-Sakov rule performs poorly in the almost degenerate case $\rho = 0.9$ and it is also not optimal in some other cases.
For these reasons we use the cluster rule in all remaining examples of this section.

We continue comparing our proposal with the Lin-Han-estimator for continuous data. To ensure a correct implementation, we use the authors' \texttt{R} code, to which they provide a link in \cite{lin_han:2022}, Remark 1.6. The results of this comparison in the Gaussian model \eqref{det2} are given in Table \ref{tab:normal}, where the left  part corresponds to the $m$ out of $n$ bootstrap using the cluster rule.
We observe that the coverage probabilities of the confidence intervals \eqref{det1}
attain values close to the desired level of $95\%$, whenever $n \geq 500$. This is in stark contrast to the confidence intervals based on the Lin-Han-estimator which are displayed in the right part of  Table \ref{tab:normal} and obtained by replacing $\sigma_{m,n}^*$ in Eq.\@ \eqref{det1} by the Lin-Han-estimate of the limiting standard deviation. These confidence intervals only seem to perform well for $n \geq 1000$, rendering them useless for small to medium sample sizes. When both methods yield satisfying covering rates, the lengths of the confidence intervals are virtually identical. On the other hand, the rRMSE values based on the Lin-Han-estimator are much larger than those based on the $m$ out of $n$ bootstrap. In light of these results, both for the rRMSE and the confidence intervals, we consider the $m$ out of $n$ method a substantial improvement on the Lin-Han-estimator.

A further comparison with the Lin-Han-estimator is made for non-Gaussian continuous data. Recall that if $U$ is $\chi^2$-distributed with $\nu$ degrees of freedom and $Z$ is a centred $p$-dimensional Gaussian vector with covariance matrix $\Sigma$, then the distribution of
$Z/\sqrt{U/\nu}$ is called  a $p$-dimensional $t(\nu)$-distribution with scale matrix $\Sigma$. For our second continuous model we consider a $2$-dimensional $t(3)$-distribution with scale matrix
\begin{equation}
    \label{eq:t3_correlation}
    \Sigma = \begin{pmatrix}
        1 & \rho \\ \rho & 1
    \end{pmatrix}.
\end{equation}
The results of this comparison are given in Table \ref{tab:t3}. Again, the $m$ out of $n$ bootstrap clearly outperforms the Lin-Han-estimator in terms of rRMSE as well as covering rates.

The results of our simulations based on discrete data were already presented in Figure \@ \ref{fig:bickel_cluster_poisson} and \ref{fig:bickel_cluster_poisson_covering}. Since there are no alternative estimation methods available in this case, we do not have a comparison for the $m$ out of $n$ bootstrap. Again, the rRMSE values seem convincing for medium to large sample sizes and the empirical coverage probabilities of the confidence intervals are satisfactory; indeed, even for the sample size  $n = 50$, the empirical coverage probabilitiy never drops below 91\%.

Let us now compare Chatterjee's variance estimator and the $m$ out of $n$ bootstrap estimator in the  case of discrete data such that $X$ and $Y$ are independent, where we again use the discrete model \eqref{det4} (with  $\rho = 0$). The results are given in Table \ref{tab:comparison-chatterjee-bs}, and we observe that Chatterjee's variance estimator has a slightly lower RMSE and yields covering rates which are very close to the desired level of 95 \%. The lengths of the confidence intervals are very similar (this is to be expected as both estimators converge to the same limiting variance). Overall, we conclude that the performance of the $m$ out of $n$ bootstrap estimator is comparable to Chatterjee's variance estimator for moderate and large sample sizes, but worse for smaller sample sizes. In particular, if one is interested in testing independence of discrete data, we see no reason to not use Chatterjee's estimation method for computing the critical values. The obvious advantage of our method, however, is that it can deal with dependent data just as well, whereas Chatterjee's variance estimator is only consistent in the independent case. In this particular simulation study, its estimates for the limiting variance were too large for all values of $\rho > 0$, resulting in a covering rate of 100\% at the cost of significantly longer confidence intervals, and strictly positive limiting values of the RMSE.

We conclude this section with a brief discussion of the time complexity of the adaptive version of the
$m$ out of $n$ bootstrap.  Recall from Remark \ref{det10} that for a given $m$ this procedure has
a time complexity   of order $\mathcal{O}(B m \log m )$, where $B$ is the number of bootstrap iterations. For the cluster rule, the $m$ out of $n$ bootstrap is performed $K$ times, where $K$ denotes the fixed number of admissible rules that are considered. Let us also assume that in the Monte-Carlo simulations performed to approximate the distributions $L_1^*, \ldots, L_K^*$ we use the same number $B$ of bootstrap repetitions. Then each $L_j^*$ is a discrete distribution with its mass concentrated in at most $B$ points. Sorting these points of mass can be done in time $\mathcal{O}\left(B \log B\right)$ \citep[for instance via heapsort; cf. Chapter 6.4 in][]{cormen_et_al:introduction_algorithms}, and constructing the cumulative distribution function from this sorted array only takes an additional $\mathcal{O}(B)$ as it only involves summation and division. Thus, the overall time of determining the cumulative distribution function of $L_j^*$ is $\mathcal{O}(B \log B)$. Now for each pair $L_i^*$ and $L_j^*$ with $i \neq j$, calculating the Kolmogorov distance $\rho\left(L_i^*, L_j^*\right)$ reduces down to comparing the corresponding cumulative distribution functions at each of their points of discontinuity, of which there are at most $2B$, and so the entire process can be done in $\mathcal{O}(2B \log B + 2B) = \mathcal{O}(B \log B)$. Thus, the entire $m$ out of $n$ bootstrap using the cluster method takes $\mathcal{O}(KBm \log m + K(K-1)B \log B)$ steps. If the parameters $K$ and $B$ stay fixed, this is certainly sub-linear in $n$; we can even allow both $K$ and $B$ to tend to infinity at a certain rate while retaining sub-linear time complexity.

For the Bickel-Sakov rule, similar arguments can be made. We have to calculate the Kolmogorov distance between $\log_{1/q} n$ pairs of distributions, where $\log_b$ denotes the logarithm for a basis $b$, and so the total time complexity is given by $\mathcal{O}(\log_{1/q} n B m \log m + \log_{1/q} n B \log B)$, which is sub-linear for fixed $q$ and $B$, and stays sub-linear even if $q \to 0$ and $B \to \infty$, provided the convergences have appropriate rates.

\section{Conclusions}
\label{sec:conclusion}
As proven by \cite{lin_han:2023} in their technically impressive article, the usual bootstrap fails for Chatterjee's rank correlation. \cite{chatterjee:2021} provides an estimator that performs well, but is limited to the case of data with independent coordinates $X$ and $Y$ and thus can be used only for testing independence. On the other hand, \cite{lin_han:2022} derive an estimator that works regardless of the dependence between $X$ and $Y$ but is limited to sample generating processes with continuous distribution functions.

In this article, we have proposed an $m$ out of $n$ bootstrap to estimate the limiting distribution of Chatterjee's rank correlation. Our method is easy to apply and by Theorem \ref{thm:mon_bs_u}, it is consistent whenever the observed statistic $T_n$ converges to some non-degenerate normal distribution. This general result is of particular interest,
because so far  there do not exist any weak convergence results
for Chatterjee's rank correlation in the discrete case
if $X$ and $Y$ are  not independent. Our result is directly applicable as soon as such a convergence is established. We have also proved  consistency
of the $m$ out of $n$ bootstrap with respect to  the Wasserstein distance $d_2$, which is a stronger result than weak convergence alone.

Simulations indicate that our method performs well in both discrete and continuous settings, even at smaller sample sizes. In terms of the RMSE, our bootstrap estimator consistently outperforms the estimation method proposed by \cite{lin_han:2022}. Comparison with Chatterjee's estimator from \cite{chatterjee:2021} is less conclusive; but as Chatterjee's estimator is only consistent for independent $X$ and $Y$, this comparison is only relevant if one wants to test independence of $X$ and $Y$. For any other application our estimation method is well-suited. It is also worth pointing out that, to the best of our knowledge, there are no alternative estimation methods when the sample generating processes are discrete but $X$ and $Y$ are not independent. The time complexity of our bootstrap scheme is sub-linear in $n$ and  the benefit in performance that the $m$ out of $n$ bootstrap yields is not paid for in excessive runtime. For applications we have  developed a data adaptive rule for choosing $m$ (the cluster method) and have demonstrated that this method  performs very well in a variety of situations.

\section{Proofs}
\label{sec:proof}
\subsection{Proof of Theorem \ref{thm:mon_bs_u}}

We begin with an auxiliary result, which will be used in the proof of Theorem \ref{thm:mon_bs_u}.
\begin{lemma}
    \label{lem:glivenko_cantelli_in_prob}
    Let $F_n$ be a sequence of random distribution functions such that for any fixed but arbitrary $x \in \mathbb{R}$ it holds that $F_n(x) \to F(x)$ in probability and $F_n(x^-) \to F(x^-)$ in probability, where $f(x^-) := \lim_{t \nearrow x} f(t)$ for a function $f$. Then $D_n := \sup_{x \in \mathbb{R}} |F_n(x) - F(x)|$ converges to $0$ in probability.
\end{lemma}
\begin{proof}
    This is essentially a version of the Glivenko-Cantelli-Theorem. We follow the proof given in \cite{billingsley:prob_and_measure}, Theorem 20.6. The main difference is that we do not assume $F_n$ to be the empirical distribution function of an iid sample, replacing the almost sure pointwise convergences by pointwise convergences in probability.

    Define the quantities $\varphi(u) := \inf\{x \in \mathbb{R} ~|~u \leq F(x)\}$ and $x_{m,k} := \varphi(k/m)$ for all $m \in \mathbb{N}$ and $k = 1, \ldots, m$. Furthermore, let
    $$
        D_{m,n} := \max_{1 \leq k \leq m} \{|F_n(x_{m,k}) - F(x_{m,k})| \lor |F_n(x_{m,k}^-) - F(x_{m,k}^-)|\},
    $$
    where $a \lor b := \max\{a,b\}$. As in the original proof of Theorem 20.6 in \cite{billingsley:prob_and_measure} one can show that $D_n \leq D_{m,n} + 1/m$ for all $m,n \in \mathbb{N}$. This is a consequence of the fact that distribution functions are isotonic and bounded. Now let $\varepsilon, \delta > 0$ be arbitrary but fixed. Then it follows that
    \begin{align}
        \begin{split}
            \label{eq:billingsley_grenze}
            \mathbb{P}(D_n > \varepsilon) &\leq \mathbb{P}\left(D_{m,n} + \frac{1}{m} > \varepsilon\right) \\
            &\leq \sum_{k=1}^m \mathbb{P}\left(|F_n(x_{m,k}) - F(x_{m,k})| > \frac{1}{2}\left(\frac{\varepsilon}{m} - m^{-2}\right)\right) \\
            &\quad + \sum_{k=1}^m \mathbb{P}\left(|F_n(x_{m,k}^-) - F(x_{m,k}^-)| > \frac{1}{2}\left(\frac{\varepsilon}{m} - m^{-2}\right)\right)
        \end{split}
    \end{align}
    for all $m,n \in \mathbb{N}$. Choose $m_0 = m_0(\varepsilon) := \min\{m \in \mathbb{N} ~|~ \varepsilon/m - 1/m^2 > 0\} \leq 1 + \lceil 1/\varepsilon\rceil$ and $n_0$ as the smallest integer such that both
    $$
        \mathbb{P}\left(|F_n(x_{m_0,k}) - F(x_{m_0,k})| > \frac{1}{2}\left(\frac{\varepsilon}{m_0} - m_0^{-2}\right)\right) < \frac{\delta}{2m_0}
    $$
    and
    $$
        \mathbb{P}\left(|F_n(x_{m_0,k}^-) - F(x_{m_0,k}^-)| > \frac{1}{2}\left(\frac{\varepsilon}{m_0} - m_0^{-2}\right)\right) < \frac{\delta}{2m_0}
    $$
    are satisfied for all $1 \leq k \leq m_0$ and $n \geq n_0$. Now Eq.\@ \eqref{eq:billingsley_grenze} implies that $\mathbb{P}(D_n > \varepsilon) < \delta$ for all $n \geq n_0$, which proves our claim.
\end{proof}

\begin{proof}[Proof of Theorem \ref{thm:mon_bs_u}]
    Throughout the proof, $Q = \mathcal{L}(X,Y)$  will denote the joint distribution of $X$ and $Y$.

    We begin by assuming that condition $(ii)$ is fulfilled and show that $T_n$ may be expressed as a functional of the sample data and their marginal distribution $Q$, and that $T_{m,n}^*$ may be expressed as a functional of the bootstrap data and the empirical distribution $Q_n$ of the observed data $(X_1, Y_1), \ldots, (X_n, Y_n)$. Without loss of generality, assume that the observations $(X_1, Y_1), \ldots, (X_n, Y_n)$ are all distinct. For any probability measure $\eta$ on some $\left(\mathbb{R}^d, \mathcal{B}\left(\mathbb{R}^d\right)\right)$, where $\mathcal{B}\left(\mathbb{R}^d\right)$ denotes the Borel-$\sigma$-algebra on $\mathbb{R}^d$, let $(U_k(\eta))_{k \in \mathbb{N}}$ be some iid process with marginal distribution $\eta$, and define a random sequence $(j(i))_{i \in \mathbb{N}}$ by $j(1) := 1$ and
    $$
        j(i) := \min\{k \in \mathbb{N} ~|~ U_k(\eta) \notin \{U_{j(1)}, \ldots, U_{j(i-1)}\}\},
    $$
    for all $i > 1$ (we may use the convention $\min \emptyset = \infty$ if $\eta$ only has mass at finitely many points). Finally, for any appropriate $k \in \mathbb{N}$ (i.e.\@ any $k$ such that $j(k) < \infty$), let
    $$
        V(k, \eta) := \left(U_{j(i)}(\eta)\right)_{1 \leq i \leq k}.
    $$

    $V(k,\eta)$ is a random vector with $k$ coordinates, each of which has marginal distribution $\eta$. By construction, all of its coordinates are distinct. If $\eta$ has a continuous distribution function, then $j(i) = i$ almost surely, and so $V(k,\eta) = (U_1(\eta), \ldots, U_k(\eta))$ almost surely. In this case, for any $\eta^k$-integrable function $f : \left(\mathbb{R}^d\right)^k \to \mathbb{R}$, it holds that
    $$
        \mathbb{E}\left[f(V(k,\eta))\right] = \mathbb{E}\left[f(U_1(\eta), \ldots, U_k(\eta))\right] = \int f ~\mathrm{d}\eta^k.
    $$
    On the other hand, if $\eta$ is a uniform distribution on a set $\{x_1, \ldots, x_N\}$, then $V(k,\eta)$ can be constructed by drawing $k$ elements from $\{x_1, \ldots, x_N\}$ without replacement, and so
    $$
        \mathbb{E}\left[f(V(k,\eta))\right] = \frac{(N - k)!}{N!}\sum_{(i_1, \ldots, i_k) \in B(N,k)} f(x_{i_1}, \ldots, x_{i_k}),
    $$
    where $B(N,k) := \left\{(i_1, \ldots, i_k) \in \{1, \ldots, N\}^k ~|~ i_j \neq i_l ~\forall j \neq l\right\}$ is the set of all collections of pairwise different indices between $1$ and $N$. If $f$ is symmetric in its arguments then this further simplifies to
    $$
        \mathbb{E}\left[f(V(k,\eta))\right] = \frac{(N-k)!}{N!} \, k! \sum_{1 \leq i_1 < \ldots < i_k \leq N} f(x_{i_1}, \ldots, x_{i_k}) = {N \choose k}^{-1} \sum_{1 \leq i_1 < \ldots < i_k \leq N} f(x_{i_1}, \ldots, x_{i_k}).
    $$
    Defining the functional $L(k,\eta) := \mathbb{E}\left[\xi_k\left(V(k,\eta)\right)\right]$, we thus obtain that
    $$
        \mathbb{E}\xi_n = \int \xi_n ~\mathrm{d}Q^n = L(n,Q),
    $$
    and
    \begin{align*}
        \mathbb{E}\left[\xi_m^* ~|~ (X_1, Y_1), \ldots, (X_n, Y_n)\right] & = {n \choose m}^{-1} \sum_{1 \leq i_1 < \ldots < i_m \leq n} \xi_m\left((X_{i_1}, Y_{i_1}), \ldots, (X_{i_m}, Y_{i_m})\right) \\
                                                                          & = L(m,Q_n),
    \end{align*}
    where $Q_n$ denotes the empirical measure of the sample $(X_1, Y_1), \ldots, (X_n, Y_n)$. Therefore,
    $$
        T_n = \sqrt{n}\left(\xi_n - L(n,Q)\right),
    $$
    and
    $$
        T_{m,n}^* = \sqrt{m}\left(\xi_m^* - L(m,Q_n)\right),
    $$
    i.e.\@ $T_n = t_n((X_1, Y_1), \ldots, (X_n, Y_n), Q)$ may be expressed as a functional of the observed sample data and their marginal distribution, and $T_{m,n}^* = t_m\left((X_1^*, Y_1^*), \ldots, (X_m^*, Y_m^*), Q_n\right)$ as a functional of the bootstrap data and the empirical distribution $Q_n$.

    Finally, for these functionals $t_m$, it holds that
    $$
        t_m((X_1, Y_1), \ldots, (X_m, Y_m), Q) - t_m((X_1, Y_1), \ldots, (X_m, Y_m), Q_n) = \sqrt{m}(L(m, Q_n) - L(m, Q)).
    $$
    We wish to bound this difference in probability. Note that
    \begin{align*}
        \mathbb{E} L(m,Q_n) & = {n \choose m}^{-1} \sum_{1 \leq i_1 < \ldots < i_m \leq n} \mathbb{E}\left[\xi_m\left((X_{i_1}, Y_{i_1}), \ldots, (X_{i_m}, Y_{i_m})\right)\right] \\
                            & = \mathbb{E}\left[\xi_m\left((X_{1}, Y_{1}), \ldots, (X_m, Y_m)\right)\right] = L(m,Q),
    \end{align*}
    and
    $$
        \mathrm{Var}(L(m,Q_n)) \leq \frac{m}{n} \, \mathrm{Var}\left(\xi_m((X_{1}, Y_{1}), \ldots, (X_{m}, Y_{m}))\right)
    $$
    by Theorem 5.2 in \cite{hoeffding:1948}. This is standard theory for $U$-statistics of iid data. By Proposition 1.2 in \cite{lin_han:2022}, we obtain that
    \begin{equation}
        \label{eq:lin_han_varianz}
        \mathrm{Var}\left(\xi_m((X_{1}, Y_{1}), \ldots, (X_{m}, Y_{m}))\right) = \mathcal{O}\left(\frac{1}{m}\right).
    \end{equation}

    We thus obtain
    $$
        L(m, Q_n) = \mathbb{E}L(m, Q_n) + \mathcal{O}_\mathbb{P}\left(\sqrt{\mathrm{Var}(L(m,Q_n))}\right) = L(m,Q) + \mathcal{O}_\mathbb{P}\left(\frac{1}{\sqrt{n}}\right),
    $$
    and so
    \begin{align*}
         & t_m((X_1, Y_1), \ldots, (X_m, Y_m), Q) - t_m((X_1, Y_1), \ldots, (X_m, Y_m), Q_n)                                                 \\
         & \qquad = \sqrt{m}\,\mathcal{O}_\mathbb{P}\left(\frac{1}{\sqrt{n}}\right) = \mathcal{O}_\mathbb{P}\left(\sqrt{\frac{m}{n}}\right),
    \end{align*}
    which is $o_\mathbb{P}(1)$ by assumption. Theorem 1 in \cite{bickel_et_al:1997} now implies
    \begin{equation}
        \label{eq:bickel_goetze_vanzweet}
        \left|\mathbb{E}\left[h\left(T_{m,n}^*\right) ~|~ (X_1, Y_1), \ldots, (X_n, Y_n)\right] - \mathbb{E}[h(T_m)]\right| \xrightarrow[n \to \infty]{\mathbb{P}} 0
    \end{equation}
    for any continuous and bounded real-valued function $h$.

    Fix some $x \in \mathbb{R}$. For any $\varepsilon > 0$, define $g_{x, \varepsilon}(t)$ as $1$ for $t \leq x - \varepsilon$, $0$ for $t > x$ and linearly interpolated between $x - \varepsilon$ and $x$. Similarly, define $h_{x, \varepsilon}(t)$ as $1$ for $t \leq x$, $0$ for $t > x + \varepsilon$ and linearly interpolated between $x$ and $x + \varepsilon$. Obviously it holds that $g_{x,\varepsilon} \leq \textbf{1}_{(-\infty,x]} \leq  h_{x,\varepsilon}$, where $\textbf{1}_A$ denotes the indicator function of a set $A$, which gives
    \begin{align}
        \begin{split}
            \label{eq:indikatorfunktionen}
            &\left|\mathbb{E}\left[\textbf{1}_{(-\infty,x]}\left(T_{m,n}^*\right) ~|~ (X_1, Y_1), \ldots, (X_n, Y_n)\right] - \mathbb{E}[\textbf{1}_{(-\infty,x]}(T_m)]\right| \\
            &\quad \leq \max_{f \in \{g_{x,\varepsilon}, h_{x, \varepsilon}\}} \left|\mathbb{E}\left[f\left(T_{m,n}^*\right) ~|~ (X_1, Y_1), \ldots, (X_n, Y_n)\right] - \mathbb{E}[\textbf{1}_{(-\infty,x]}(T_m)]\right|.
        \end{split}
    \end{align}
    An application of the triangle inequality yields
    \begin{align*}
         & \left|\mathbb{E}\left[g_{x,\varepsilon}\left(T_{m,n}^*\right) ~|~ (X_1, Y_1), \ldots, (X_n, Y_n)\right] - \mathbb{E}[\textbf{1}_{(-\infty,x]}(T_m)]\right|                                                                                                                             \\
         & \quad \leq \left|\mathbb{E}\left[g_{x,\varepsilon}\left(T_{m,n}^*\right) ~|~ (X_1, Y_1), \ldots, (X_n, Y_n)\right] - \mathbb{E}[g_{x, \varepsilon}(T_m)]\right| + \left|\mathbb{E}\left[g_{x,\varepsilon}\left(T_{m}\right)\right] - \mathbb{E}[\textbf{1}_{(-\infty,x]}(T_m)]\right|.
    \end{align*}
    The first term converges to $0$ in probability by Eq.\@ \eqref{eq:bickel_goetze_vanzweet}, and the second term converges by assumption to $\left|\mathbb{E}\left[g_{x,\varepsilon}\left(\sigma Z\right)\right] - \mathbb{E}[\textbf{1}_{(-\infty,x]}(\sigma Z)]\right|$, where $Z$ denotes a standard normally distributed random variable. Because $|g_{x,\varepsilon} - \textbf{1}_{(-\infty,x]}| \leq 1$, it holds that
    $$
        \left|\mathbb{E}\left[g_{x,\varepsilon}\left(\sigma Z\right)\right] - \mathbb{E}[\textbf{1}_{(-\infty,x]}(\sigma Z)]\right| \leq \mathbb{P}\left(x-\varepsilon < \sigma Z \leq x\right) \leq \frac{\varepsilon}{\sqrt{2\pi}\sigma}.
    $$
    Similar arguments can be made for the function  $h_{x,\varepsilon}$, and so
    $$
        \mathbb{P}\Big (\max_{f \in \{g_{x,\varepsilon}, h_{x, \varepsilon}\}} \left\{\left|\mathbb{E}\left[f\left(T_{m,n}^*\right) ~|~ (X_1, Y_1), \ldots, (X_n, Y_n)\right] - \mathbb{E}[\textbf{1}_{(-\infty,x]}(T_m)]\right|\right\} > \frac{\varepsilon}{\sqrt{2\pi} \sigma}\Big ) \xrightarrow[n \to \infty]{} 0.
    $$
    $\varepsilon > 0$ was arbitrary, and so Eq.\@ \eqref{eq:indikatorfunktionen} implies that
    $$
        \left|\mathbb{E}\left[\textbf{1}_{(-\infty,x]}\left(T_{m,n}^*\right) ~|~ (X_1, Y_1), \ldots, (X_n, Y_n)\right] - \mathbb{E}[\textbf{1}_{(-\infty,x]}(T_m)]\right| \xrightarrow[n \to \infty]{\mathbb{P}} 0.
    $$
    By changing the definitions of $g_{x,\varepsilon}$ and $h_{x,\varepsilon}$ slightly, the same convergence can be shown for $\textbf{1}_{(-\infty, x)}$. Since $x \in \mathbb{R}$ was an arbitrary point, we have thus verified the conditions of Lemma \ref{lem:glivenko_cantelli_in_prob}, which proves the assertion of Theorem \ref{thm:mon_bs_u} under condition (ii).

    Now suppose that condition $(i)$ is satisfied. Define the objects
    \begin{align}
        \begin{split}
            \label{eq:def_u_v_t}
            U_{m,n} &:= {n \choose m}^{-1} \sum_{1 \leq i_1 < \ldots < i_m \leq n} \xi_m((X_{i_1}, Y_{i_1}), \ldots, (X_{i_m}, Y_{i_m})), \\
            V_{m,n} &:= n^{-m} \sum_{1 \leq i_1, \ldots, i_m \leq n} \xi_m((X_{i_1}, Y_{i_1}), \ldots, (X_{i_m}, Y_{i_m})), \\
            \tilde{T}_{m,n}^* &:= \sqrt{m}\left(\xi_m^* - \int \xi_m ~\mathrm{d}Q_n^m\right).
        \end{split}
    \end{align}
    Note that $U_{m,n} = \mathbb{E}\left[\xi_m^* ~|~ (X_1, Y_1), \ldots, (X_n, Y_n)\right]$ and $V_{m,n} = \int \xi_m ~\mathrm{d}\xi_n^m$. Thus,
    \begin{equation}
        \label{eq:tildeschlangestern_tildestern}
        \tilde{T}_{m,n}^* = T_{m,n}^* + \sqrt{m}\left(U_{m,n} - V_{m,n}\right),
    \end{equation}

    Define the sets $B(k,n) := \{(i_1, \ldots, i_m) ~|~ 1 \leq i_1, \ldots, i_m \leq n \land \#\{i_1, \ldots, i_m\} = k\}$. Each $B(k,n)$ is the set of all collections of indices $(i_1, \ldots, i_m)$ with exactly $k$ unique values between $1$ and $n$. The $V$-statistic $V_{m,n}$ can be represented as
    \begin{align*}
        V_{m,n} & = n^{-m}\sum_{k=1}^m \sum_{(i_1, \ldots, i_m) \in B(k,n)} \xi_m((X_{i_1}, Y_{i_1}), \ldots, (X_{i_m}, Y_{i_m})) \\
                & = \sum_{k=1}^m \frac{n!}{n^m \, (n-k)!} U_{k,n},
    \end{align*}
    where $U_{k,n}$ is a $U$-statistic of order $k$. In the special case $k = m$, we obtain the $U$-statistic $U_{m,n}$ from Eq.\@ \eqref{eq:def_u_v_t}. Note that
    $$
        \frac{n!}{n^m (n-k)!} = \frac{n}{n} \cdot \frac{n-1}{n} \cdots \frac{n-k+1}{n} \cdot n^{-(m-k)} \leq n^{-(m-k)}
    $$
    for any $1 \leq k \leq m$. Since $\mathbb{E}U_{m,n} = \mathbb{E}\left[\xi_m((X_1, Y_1), \ldots, (X_m, Y_m))\right]$, it holds that
    \begin{align}
        \begin{split}
            \label{eq:ew_vn}
            \left|\mathbb{E}[V_{m,n}] - \mathbb{E}\left[\xi_m((X_1, Y_1), \ldots, (X_m, Y_m))\right]\right| &= \left|\sum_{k=1}^{m-1} \mathbb{E}U_{k,n}\right| \\
            &\leq \sum_{k=1}^{m-1} n^{-(m-k)} \mathbb{E}\left|U_{k,n}\right| \\
            &\leq \frac{m}{n},
        \end{split}
    \end{align}
    because Chatterjee's rank correlation is absolutely bounded by $1$, and so $\mathbb{E}\left|U_{k,n}\right| \leq 1$ for all $k$. Furthermore,
    \begin{align*}
        \mathrm{Var}(V_{m,n}) & = \sum_{k=1}^m \mathrm{Var}\left(\frac{(n-k)!}{n^m \, n!} U_{k,n}\right) + 2\sum_{1 \leq i < j \leq m} \mathrm{Cov}\left(\frac{(n-i)!}{n^m \, n!} U_{i,n}, \frac{(n-j)!}{n^m \, n!} U_{j,n}\right) \\
                              & \leq \mathrm{Var}(U_{m,n}) + \sum_{k=1}^{m-1} n^{-2(m-k)} \mathrm{Var}(U_{k,n}) + 2 \sum_{1 \leq i < j \leq m} n^{-2m + i + j} \left|\mathrm{Cov}(U_{i,n}, U_{j,n})\right|                         \\
                              & \leq \mathrm{Var}(U_{m,n}) + 2 \sum_{k=1}^{m-1} \frac{k}{n^{2(m+k) + 1}} + 4 \sum_{1 \leq i<j \leq m} n^{-2m+i+j} \sqrt{\frac{ij}{n^2}}                                                            \\
                              & \leq 2\frac{m}{n} + 2\frac{m^2}{n^3} + 4 \sum_{1 \leq i<j \leq m} \frac{m}{n^{2m+1-i-j}},
    \end{align*}
    where we have again used the fact that, for any $U$-statistic $U$ of order $k$ with kernel $h$ and sample data $W_1, \ldots, W_n$, it holds by Theorem 5.2 in \cite{hoeffding:1948} that
    $$
        \mathrm{Var}(U) \leq \frac{k}{n} \, \mathrm{Var}(h(W_1, \ldots, W_k)),
    $$
    and if $h \leq c$ almost surely, then $\mathrm{Var}(h(W_1, \ldots, W_k)) \leq 2c^2$. In this case, $c = 1$. Now,
    \begin{align*}
        \sum_{1 \leq i<j \leq m} \frac{m}{n^{2m+1-i-j}} & = \sum_{s=3}^{2m-1} \left\lfloor \frac{s-1}{2}\right\rfloor \frac{m}{n^{2m+1 - s}} \leq \sum_{s=3}^{2m-1} \frac{m^2}{n^{2m+1-s}}              \\
                                                        & = \frac{m^2}{n^2} + \sum_{s=3}^{2m-2} \frac{m^2}{n^{2m+1-s}} \leq \frac{m^2}{n^2} + 2m\,\frac{m^2}{n^3} = \frac{m^2}{n^2} + 2\frac{m^3}{n^3}.
    \end{align*}
    Thus,
    \begin{equation}
        \label{eq:var_vn}
        \mathrm{Var}(V_{m,n}) \leq 2\frac{m}{n} + 2\frac{m^2}{n^3} + 4\frac{m^2}{n^2} + 8\frac{m^3}{n^3} \leq 16 \, \frac{m}{n}
    \end{equation}
    for almost all $n$, since $m = o\left(\sqrt{n}\right)$ by assumption. Eqs.\@ \eqref{eq:ew_vn} and \eqref{eq:var_vn} now give us
    \begin{align}
        \begin{split}
            \label{eq:varianz_vmn_ew}
            V_{m,n} &= \mathbb{E}\left[\xi_m((X_1, Y_1), \ldots, (X_m, Y_m))\right] + \frac{m}{n} + \mathcal{O}_\mathbb{P}\left(\sqrt{\frac{m}{n}}\right) \\
            &= \mathbb{E}\left[\xi_m((X_1, Y_1), \ldots, (X_m, Y_m))\right] +  \mathcal{O}_\mathbb{P}\left(\sqrt{\frac{m}{n}}\right).
        \end{split}
    \end{align}
    Standard theory for $U$-statistics implies that
    $$
        U_{m,n} = \mathbb{E}\left[\xi_m((X_1, Y_1), \ldots, (X_m, Y_m))\right] +  \mathcal{O}_\mathbb{P}\left(\sqrt{\frac{m}{n}}\right),
    $$
    and so
    $$
        |U_{m,n} - V_{m,n}| = \mathcal{O}_\mathbb{P}\left(\sqrt{\frac{m}{n}}\right).
    $$
    By Eq.\@ \eqref{eq:tildeschlangestern_tildestern}, we therefore get
    \begin{equation}
        \label{eq:tmn_tilde_stern_prob_konv}
        \tilde{T}_{m,n}^* = T_{m,n}^* + o_\mathbb{P}(1),
    \end{equation}
    and it suffices to prove weak convergence in probability of $\tilde{T}_{m,n}^*$.

    The statistic $\tilde{T}_{m,n}^*$ obviously enjoys an expression as a functional of the bootstrap data and the empirical measure, $\tilde{T}_{m,n}^* = \tilde{t}_m\left((X_1^*, Y_1^*), \ldots, (X_{m}^*, Y_m^*), Q_n\right)$. Analogously, it holds for the observed statistic that $T_n = \tilde{t}_n((X_1, Y_1), \ldots, (X_n, Y_n), Q)$. We again need to bound the difference
    \begin{align}
        \begin{split}
            \label{eq:variance_difference_1}
            &\tilde{t}_m((X_1, Y_1), \ldots, (X_m, Y_m), Q_n) - \tilde{t}_m((X_1, Y_1), \ldots, (X_m, Y_m), Q) \\
            &\quad= \sqrt{m} \left( \int \xi_m ~\mathrm{d}Q^m - \int \xi_m ~\mathrm{d}Q_n^m\right) \\
            &\quad= \sqrt{m}\left(\mathbb{E}\left[\xi_m((X_1, Y_1), \ldots, (X_m, Y_m))\right] - V_{m,n}\right).
        \end{split}
    \end{align}
    We know from Eq.\@ \eqref{eq:varianz_vmn_ew} that the last term is $\sqrt{m} \, \mathcal{O}_\mathbb{P}\left(\sqrt{m/n}\right) = \mathcal{O}_\mathbb{P}\left(m/\sqrt{n}\right)$, which is $o_\mathbb{P}(1)$ by assumption. Theorem 1 in \cite{bickel_et_al:1997} now once more implies
    \begin{equation}
        \label{eq:bickeletal_tschlange}
        \left|\mathbb{E}\left[h\left(\tilde{T}_{m,n}^*\right) ~|~(X_1, Y_1), \ldots, (X_n, Y_n)\right] - \mathbb{E}[h(T_m)]\right| \xrightarrow[n \to \infty]{\mathbb{P}} 0
    \end{equation}
    for any continuous and bounded real-valued function $h$. We can now proceed as before to show that the assumptions of Lemma \ref{lem:glivenko_cantelli_in_prob} are satisfied, which proves  together with Eq.\@ \eqref{eq:tmn_tilde_stern_prob_konv} the assertion of Theorem \ref{thm:mon_bs_u} under condition (i).
\end{proof}

\subsection{Proof of Theorem \ref{thm:konvergenz_wasserstein_metrik}}
\label{subsec:proof_thm_wasserstein}

We state and prove three auxiliary results, which will be essential in the proof of Theorem \ref{thm:konvergenz_wasserstein_metrik}.

\begin{lemma}
    \label{lem:wasserstein_permutation}
    Let $X,Y $ be two $\mathcal{S}$-valued
    random variables and $Z : \mathcal{S} \to \mathcal{S}$ a random function such that $Z$ is independent of $X$ and $z(X)$ has the same distribution as $X$  for $\mathcal{L}(Z)$-almost all $z$. Then the set of couplings of $Z(X)$ and $Y$ is identical to the set of couplings of $X$ and $Y$.
\end{lemma}
\begin{proof}
    Let $M \subseteq \mathcal{S}$ be a measurable set. Then
    $$
        \mathbb{P}(Z(X) \in M) = \iint \textbf{1}_M(z(x)) ~\mathrm{d}\mathbb{P}^X(x) ~\mathrm{d}\mathbb{P}^Z(z) = \iint \textbf{1}_M(x) ~\mathrm{d}\mathbb{P}^X(x) ~\mathrm{d}\mathbb{P}^Z(z) = \mathbb{P}(X \in M).
    $$
    Therefore $Z(X)$ has the same distribution as $X$, and so the set of all couplings of $Z(X)$ and $Y$ is identical to the set of all couplings of $X$ and $Y$.
\end{proof}

\begin{lemma}
    \label{lem:darstellungslemma}
    Let $(\mathcal{X}, \mathcal{B}(\mathcal{X}))$ and $(\mathcal{Y}, \mathcal{B}(\mathcal{Y}))$ be two Borel spaces and $f : \mathcal{X} \to \mathcal{Y}$ a measurable function. If $Y$ is a $\mathcal{Y}$-valued random variable such that $\mathcal{L}(Y) = \mathcal{L}(f(X))$ for some $\mathcal{X}$-valued random variable $X$, then there exists a copy $X'$ of $X$ such that $Y = f(X')$ almost surely.
\end{lemma}
\begin{proof}
    Let $\mu$ and $\nu$ denote the distributions of $X$ and $Y$, respectively. Then, by Theorem 1 in \cite{chang_pollard:1997}, there exists an $f$-disintegration of $\mu$, i.e.\@ a $\nu$-almost surely uniquely determined collection $(\mu_y)_{y \in \mathcal{Y}}$ of probability measures on $\mathcal{X}$ (equipped with its Borel $\sigma$-algebra) such that $\mu_y(f \neq y) = 0$ for $\nu$-almost all $y \in \mathcal{Y}$, and for all non-negative $g:\mathcal{X} \to \mathbb{R}_{\geq 0}$ it holds that $y \mapsto \int g ~\mathrm{d}\mu_y$ is measurable and $\int g ~\mathrm{d}\mu = \iint g ~\mathrm{d}\mu_y ~\mathrm{d}\nu(y)$. Note that in  general the disintegrating measures $\mu_y$ are only $\sigma$-finite. The fact that they are probability measures is a consequence of Theorem 2 in \cite{chang_pollard:1997}. The disintegration measures are regular conditional probabilities with the nice property that they concentrate on the level sets $\{f = y\}$. By Kolmogorov's existence theorem \citep[e.g.\@ Theorem 36.1 in][]{billingsley:prob_and_measure} we can construct the product measure $\xi := \bigotimes_{y \in \mathcal{Y}} \mu_y$ on the measurable space $\left(\prod_{y \in \mathcal{Y}}\mathcal{X}, \bigotimes_{y \in \mathcal{Y}} \mathcal{B}(\mathcal{X})\right)$. Let $U$ denote a random variable with distribution $\xi$ that is independent of $Y$, and for any $y \in \mathcal{Y}$ let $\pi_y : \prod_{y \in \mathcal{Y}} \mathcal{X} \to \mathcal{X}$ denote the projection onto the $y$-coordinate. Then we define $X' := \pi_Y(U)$, i.e.\@ the projection of $U$ onto the random coordinate $Y$. Then, for any measurable $A \in \mathcal{B}(\mathcal{X})$ it holds that
    \begin{align*}
        \mathbb{P}(X' \in A) & = \iint \textbf{1}_A(\pi_y(u)) ~\mathrm{d}\xi(u) ~\mathrm{d}\nu(y)
        = \int \textbf{1}_A(v) ~\mathrm{d}\mu_y(v) ~\mathrm{d}\nu(y)
        = \int \textbf{1}_A ~\mathrm{d}\mu = \mathbb{P}(X \in A),
    \end{align*}
    where the first equation holds due to Fubini's theorem, the second equation holds because $\xi$ has marginals $\mu_y, y \in \mathcal{Y},$ by construction and the third equation holds because $(\mu_y)_{y \in \mathcal{Y}}$ forms a disintegration of $\mu$. $X'$ is therefore a copy of $X$. On the other hand, writing $N := \{(x,y) \in \mathcal{X} \times \mathcal{Y} ~|~ f(x) \neq y\}$ and $M(y) := \{f \neq y\}$, it holds that
    \begin{align*}
        \mathbb{P}(f(X') \neq Y) & = \iint \textbf{1}_N(\pi_y(u), y) ~\mathrm{d}\xi(u)~\mathrm{d}\nu(y)
        = \iint \textbf{1}_{M(y)}(\pi_y(u)) ~\mathrm{d}\xi(u) ~\mathrm{d}\nu(y)                         \\
                                 & = \iint \textbf{1}_{M(y)}(v) ~\mathrm{d}\mu_y(v) ~\mathrm{d}\nu(y)
        = \int 0 ~\mathrm{d}\nu(y) = 0,
    \end{align*}
    where we have again used Fubini's theorem for the first equation, the construction of $N$ and $M(y)$ in the second equation, the fact that $\xi$ has marginals $\mu_y, y \in \mathcal{Y},$ in the third equation, and the concentration of each disintegration measure $\mu_y$ onto the level sets $\{f = y\}$ in the fourth equation.
\end{proof}

As a consequence of Lemma \ref{lem:darstellungslemma}, we have the following corollary.
\begin{corollary}
    \label{cor:wasserstein_inf}
    Let $f :\mathcal{S}' \to \mathcal{S}$ be a function between two Borel spaces $\mathcal{S}, \mathcal{S}'$. Let $X$ and $Z$ be two random variables with values in $\mathcal{S}'$ and $\mathcal{S}$, respectively. Then, for any $p \geq 1$, it holds that
    \begin{equation}
        \label{eq:wasserstein_inf_1}
        d_p^p\left(\mathcal{L}(f(X)), \mathcal{L}(Z)\right) = \inf \mathbb{E}\left[|f(X') - Z'|^p\right],
    \end{equation}
    where the infimum ranges over all vectors $(X', Z')$ whose marginals $X'$ and $Z'$ have distributions $\mathcal{L}(X)$ and $\mathcal{L}(Z)$, respectively.
\end{corollary}
\begin{proof}
    By definition of the Wasserstein distance, we have
    \begin{equation}
        \label{eq:wasserstein_inf_2}
        d_p^p\left(\mathcal{L}(f(X)), \mathcal{L}(Z)\right) = \inf \mathbb{E}\left[|Y' - Z'|^p\right],
    \end{equation}
    where the infimum is taken  over all random vectors $(Y', Z')$ with marginal distributions $\mathcal{L}(f(X))$ and $\mathcal{L}(Z)$. For any $\varepsilon > 0$ we can find a vector $(Y'_\varepsilon, Z'_\varepsilon)$ such that $\|Y'_\varepsilon - Z'_\varepsilon\|_p^p \leq d_p^p\left(\mathcal{L}(f(X)), \mathcal{L}(Z)\right) + \varepsilon$. By Lemma \ref{lem:darstellungslemma} we can find a random variable $X'_\varepsilon$ such that $Y'_\varepsilon = f(X'_\varepsilon)$ almost surely, and so it holds that $\|f(X'_\varepsilon) - Z'_\varepsilon\|_p^p \leq d_p^p\left(\mathcal{L}(f(X)), \mathcal{L}(Z)\right) + \varepsilon$. This implies
    $$
        \inf \mathbb{E}\left[|f(X') - Z'|^p\right] \leq d_p^p\left(\mathcal{L}(f(X)), \mathcal{L}(Z)\right),
    $$
    since $\varepsilon > 0$ was arbitrary. On the other hand, the reverse inequality holds because if $X'$ has distribution $\mathcal{L}(X)$, then $f(X')$ has distribution $\mathcal{L}(f(X))$, and so the infimum in Eq.\@ \eqref{eq:wasserstein_inf_1} is taken over a subset of that over which the infimum in Eq.\@ \eqref{eq:wasserstein_inf_2} is taken.
\end{proof}

\begin{proof}[Proof of Theorem \ref{thm:konvergenz_wasserstein_metrik}]
    Note that
    $$
        T_{m,n}^* = \sqrt{m}\left(\xi_m^* - \mathbb{E}\left[\xi_m^* ~|~(X_1, Y_1), \ldots, (X_n, Y_n)\right]\right)
    $$
    can be constructed in three steps. First, we generate sample data $(X_1, Y_1), \ldots, (X_n, Y_n)$. For ease of notation let us write $Z_k := (X_k, Y_k)$. Next, we draw (independently from the observed sample) a random permutation $s$ operating on $Z := (Z_1, \ldots, Z_n)$. Finally, we evaluate the statistic $T_m$ on the first $m$ observation of the randomly permuted sample. Formally, we have
    \begin{equation}
        \label{eq:bs_funktion_darstellung}
        T_{m,n}^* = \sqrt{m}\left(T_m(\pi_m(s(Z))) - u(s(Z))\right) =: f(s(Z)),
    \end{equation}
    where $\pi_m$ denotes the projection on the first $m$ coordinates, $s(Z)$ denotes the permuted sample (according to the random permutation $s$) and $u(Z)$ is the $U$-statistic with kernel $\xi_m$ and sample data $Z$ (due to the symmetry of $\xi_m$ it holds that $u(Z) = u(s(Z))$). For any two random variables $U$ and $V$ let $\Gamma(U,V)$ denote the set of their couplings. Then,
    \begin{align}
        \label{eq:wasserstein_tm_tmstern}
        \begin{split}
            d_2^2\left(\mathcal{L}\left(T_{m,n}^*\right), \mathcal{N}\left(0, \sigma^2\right)\right) &= d_2^2\left(\mathcal{L}(f(s(Z))), \mathcal{N}\left(0, \sigma^2\right)\right) = \inf_{(U,V) \in \Gamma_1} \mathbb{E}\left[|f(U)-V|^2\right] \\
            &= \inf_{(U,V) \in \Gamma_2} \mathbb{E}\left[|f(U)-V|^2\right] = d_2^2(\mathcal{L}(f(Z)), \mathcal{N}\left(0, \sigma^2\right)),
        \end{split}
    \end{align}
    where $\Gamma_1 := \Gamma(\mathcal{L}(s(Z)), \mathcal{N}\left(0, \sigma^2\right))$ is the set of all couplings of $\mathcal{L}(s(Z))$ and $\mathcal{N}\left(0, \sigma^2\right)$, $\Gamma_2 := \Gamma(\mathcal{L}(Z), \mathcal{N}\left(0, \sigma^2\right))$ is the set of all couplings of $\mathcal{L}(Z)$ and $\mathcal{N}\left(0, \sigma^2\right)$, and the first four equations hold due to Eq.\@ \eqref{eq:bs_funktion_darstellung}, Corollary \ref{cor:wasserstein_inf}, Lemma \ref{lem:wasserstein_permutation} and again Corollary \ref{cor:wasserstein_inf} (in that order). In order to prove that the right-hand side in \eqref{eq:wasserstein_tm_tmstern} converges to $0$ note that we have already shown in the proof of Theorem \ref{thm:mon_bs_u} that
    \begin{equation}
        \label{eq:l2_abstand_fz_tm}
        \|f(Z) - T_m\|_{L_2}^2 = \mathrm{Var}\left(\sqrt{m} \,\mathbb{E}\left[\xi_m^* ~|~ (X_1, Y_1), \ldots, (X_n, Y_n)\right]\right) = o(1).
    \end{equation}
    Furthermore, since the conditions of Theorem \ref{thm:mon_bs_u} are satisfied, $T_m$ converges weakly to $\mathcal{N}\left(0,\sigma^2\right)$. By assumption  $(T_m)_{m \in \mathbb{N}}$ is uniformly square-integrable,  and  Theorem 6.9 in \cite{villani:optimaltransport} implies that $T_m$ converges to $\mathcal{N}\left(0,\sigma^2\right)$ in the Wasserstein metric $d_2$. Thus, by Eq.\@ \eqref{eq:l2_abstand_fz_tm}, the right-hand side of Eq.\@ \eqref{eq:wasserstein_tm_tmstern} converges to $0$, which proves the assertion.
\end{proof}

\subsection{Proof of Remark \ref{wassersteinremark}}

The proof is a consequence of Remark \ref{rem1} $(i)$ and the following lemma.

\begin{lemma}
    \label{lem:asymptotisch_normal_lemma}
    Let $(X_n)_{n \in \mathbb{N}}$ be a sequence of random variables such that $X_n/\sigma_n \to \mathcal{N}(0,1)$ in distribution, where $\sigma_n^2 := \mathrm{Var}(X_n)$, and $X_n \to \mathcal{N}\left(0,\sigma^2\right)$ in distribution. Then $(X_n)_{n \in \mathbb{N}}$ is uniformly square-integrable.
\end{lemma}
\begin{proof}
    As $X_n$ weakly converges to $\mathcal{N}\left(0,\sigma^2\right)$, it suffices to show that $\sigma_n \to \sigma$ since this will imply convergence in the Wasserstein distance $d_2$ by Theorem 6.9 in \cite{villani:optimaltransport}, which by the same theorem is equivalent to weak convergence and uniform square-integrability. Now assume that $\sigma_n$ does not converge to $\sigma$ and fix some $\varepsilon > 0$ and a subsequence $(\sigma_{n_k})_{k \in \mathbb{N}}$ such that either $\sigma_{n_k} < \sigma - \varepsilon$ or $\sigma_{n_k} > \sigma + \varepsilon$ for all $k \in \mathbb{N}$. Without loss of generality suppose the subsequence satisfies the first inequality (as the other case can be treated analogously), then for any $z \in \mathbb{R}$
    $$
        \mathbb{P}(X_{n_k}/\sigma_{n_k} \leq z) = \mathbb{P}(X_{n_k} \leq \sigma_{n_k} z) \leq \mathbb{P}(X_{n_k} \leq (\sigma - \varepsilon)z) \xrightarrow[k \to \infty]{} \mathbb{P}(\sigma Z \leq (\sigma + \varepsilon) z),
    $$
    where $Z$ is some standard normal random variable. But by assumption, the left-hand side must converge to $\mathbb{P}(Z \leq z)$, which is a contradiction.
\end{proof}

\section*{Acknowledgements}
This work was partially supported by the DFG Research unit 5381 \textit{Mathematical Statistics in the Information Age}, project number 460867398. The authors would like to thank three unknown referees for their constructive comments on an earlier version of this paper.

\bibliographystyle{abbrvnat}
\bibliography{chatterjee_bibliography}
\end{document}